\documentclass[10pt]{article}
\usepackage{latexsym,amsfonts,amssymb,amsthm,amsmath,graphics,epsfig,extarrows}
\usepackage{mathtools}
\usepackage{arydshln,leftidx,mathtools}
\allowdisplaybreaks
\usepackage{comment} %%%% to use \begin{comment} 
\usepackage{cancel}  %%%% to use cross-out \cancelto 
\usepackage[latin1]{inputenc}
\usepackage[T1]{fontenc}
\usepackage{tikz}
\usepackage{mathrsfs}  

\usepackage{url,doi}
\tikzset{
	treenode/.style = {shape=rectangle, rounded corners,
		draw, align=center,
		top color=white,
		bottom color=blue!20},
	root/.style     = {treenode, font=\Large,
		bottom color=red!30},
	env/.style      = {treenode, font=\ttfamily\normalsize},
	dummy/.style    = {circle,draw}
}

\usepackage{enumitem}  % change enumerate bullet styles

\usepackage{todonotes}

\usetikzlibrary{arrows,shapes}
\usetikzlibrary{decorations.pathmorphing}
\usetikzlibrary{calc}

\usepackage[margin=1.2in]{geometry}

\definecolor{hred}{rgb}{1,0,0}
\definecolor{hivory}{rgb}{1,1,0.94}
\definecolor{hdarkblueblack}{RGB}{30,30,40}
\definecolor{hdarkgrayblack}{RGB}{35,35,35}
\definecolor{hdarkgrayblack}{RGB}{30,30,30}
\definecolor{hdarkgrayblack}{RGB}{25,25,25}
\usepackage[title,titletoc]{appendix}

\usepackage{empheq}
\usepackage{cases} 

\newtheorem{theorem}{Theorem}[section]
\newtheorem{lemma}[theorem]{Lemma}
\newtheorem{proposition}[theorem]{Proposition}
\newtheorem{corollary}[theorem]{Corollary}
\newtheorem{definition}{Definition}[section]

\newtheorem{assumption}[theorem]{Assumption}

\theoremstyle{remark}
\newtheorem{remark}{Remark}[section]

\numberwithin{equation}{section}
\numberwithin{theorem}{section}

\usepackage{everysel}

\EverySelectfont{%
	\fontdimen2\font=0.4em% interword space
	\fontdimen3\font=0.2em% interword stretch
	\fontdimen4\font=0.1em% interword shrink
	\fontdimen7\font=0.1em% extra space
	\hyphenchar\font=`\-% to allow hyphenation
}

\usepackage{dsfont}
\usepackage{cases}

\newcommand{\calA}{{\mathcal{A}}}

\newcommand{\calD}{{\mathcal{D}}}

\newcommand{\calL}{{\mathcal{L}}}

\newcommand{\calM}{{\mathcal{M}}}
\newcommand{\calG}{{\mathcal{G}}}
\newcommand{\calE}{{\mathcal{E}}}
\newcommand{\calV}{{\mathcal{V}}}

\newcommand{\calF}{{\mathcal{F}}}

\newcommand{\BA}{\mathbb{A}}
\newcommand{\BBA}{{\bf{A}}}
\newcommand{\BR}{\mathbb{R}}
\newcommand{\BV}{\mathbb{V}}

\newcommand{\BX}{\mathbb{X}}
\newcommand{\BBX}{{\bf{X}}}
\newcommand{\BF}{\mathbb{F}}
\newcommand{\BFF}{\mathfrak{F}}
\newcommand{\BBF}{{\bf{F}}}
\newcommand{\BB}{\mathbb{B}}
\newcommand{\BBB}{{\bf{B}}}

\newcommand{\BT}{\mathbb{T}}
\newcommand{\BU}{\mathbb{U}}
\newcommand{\BP}{\mathbb{P}}
\newcommand{\BBP}{{\bf{P}}}
\newcommand{\BS}{\mathbb{S}}
\newcommand{\BBS}{{\bf{S}}}

\newcommand{\BK}{\mathbb{K}}
\newcommand{\BCK}{\mathcal{K}}
\newcommand{\BE}{\mathbb{E}}
\newcommand{\BL}{\mathbb{L}}
\newcommand{\BBL}{{\bf{L}}}
\newcommand{\BH}{\mathbb{H}}
\newcommand{\BBH}{{\bf{H}}}

\newcommand{\BBY}{{\bf{Y}}}

\newcommand{\B}[1]{{\bf #1}}

\newcommand{\kk}[1]{_{#1=1}^m}

\makeatletter

\def\Tr@nsmogrify#1#2.{\expandafter\newcommand\csname #1#2\endcsname
	{\mathchardef\Tr@ns@temp=\mathcode\lccode`#1\relax
		\mathcode\lccode`#1=\mathcode`#1\lowercase{\csname#1#2\endcsname}%
		\mathcode\lccode`#1=\Tr@ns@temp\relax}}

\@for\x:=Sin,Cos,Max,Min,Lim,Limsup,Liminf,Inf\do{%
	\expandafter\Tr@nsmogrify \x.}

\makeatother

\begin{document}
	\title{Optimal boundary control of the isothermal semilinear Euler equation for gas dynamics on a network\footnote{M.B and M.H. acknowledge support by the Berlin Cluster of Excellence Math+ under project AA4-3; and M.H. further acknowledges support by the DFG via SFB-TRR 154 project B02}}
	
	\author{
		Marcelo Bongarti$^1$ and Michael Hinterm\"{u}ller$^{1,2}$
	}
	\date{\small %
		$^1$Weierstrass Institute for Applied Analysis and Stochastics, Berlin, Germany\\
        $^2$Humboldt-Universit\"at zu Berlin, Germany\\%
		\	\\
		\today
	}
	\maketitle
	\begin{abstract} 
		The analysis and boundary optimal control of the nonlinear transport of gas on a network of pipelines is considered. The evolution of the gas distribution on a given pipe is modeled by an isothermal semilinear compressible Euler system in one space dimension. On the network, solutions satisfying (at nodes) the so called Kirchhoff flux continuity conditions are shown to exist in a neighborhood of an equilibrium state. The associated nonlinear optimization problem then aims at steering such dynamics to a given target distribution by means of suitable (network) boundary controls while keeping the distribution within given (state) constraints. The existence of local optimal controls is established and a corresponding Karush-Kuhn-Tucker (KKT) stationarity system with an almost surely non-singular Lagrange multiplier is derived.
	\end{abstract} 
	\noindent{\bf Keywords:} optimal boundary control, gas dynamics, gas networks, isothermal Euler equation, compressible fluid dynamics, nonlinear hyperbolic PDE's, pointwise state constraints, non-singular Lagrange multiplier

%\tableofcontents

\section{Introduction}

In this paper we are interested in the analysis of optimal control problems for gas transport on a network of pipelines. For the sake of an introductory discussion, let us start by sketching the underlying state system in the simple case where the pipeline system consists of a single pipe only. Further, we state the optimization problem of interest. Once these mathematical formulations are at hand, we will set it in context in view of the currently ongoing energy transition from fossil fuels to all renewable energy sources.

Mathematically, a single pipe can be associated with the domain $\Omega=(0,L)$ where $L>0$ denotes the length of the pipe. Given physical parameters $c, \alpha, \lambda, D$ and $g$, which will be discussed in detail later, for modeling the gas transport in that pipe consider the semilinear hyperbolic initial boundary value problem (IBVP): 
\begin{subnumcases}{\label{iso2ex}}
	\dfrac{1}{c^2}\dfrac{\partial p}{\partial t} + \dfrac{\partial q}{\partial x} = 0, \label{eiso2aex} \\[2mm]
	\dfrac{\partial q }{\partial t} + \dfrac{\partial p}{\partial x} = -\dfrac{\lambda}{2D}\dfrac{q|q|}{p} - \dfrac{g\sin(\alpha)p}{c^2},  \label{eiso2bex} \\[2mm]
	p(0,x) = p_0, \ q(0,x) = q_0(x), \label{iso2cex} \\[2mm]
	p(t,0) = \Phi_1(t), \ q(t,L) = \Phi_2(t). \label{iso2dex}
\end{subnumcases} 
This system is known as the isothermal Euler system for gas dynamics. It describes the evolution in time of a given initial gas pressure $p_0: \Omega \to \BR$ and gas flux distribution $q_0: \Omega \to \BR$ along the fixed pipe $\Omega$. Within the scope of this work, the quantities $\Phi_1, \Phi_2: (0,T) \to \BR$ acting on the boundary of the pipe domain $\Omega$ take the roles of \emph{controls} which can be influenced in order to steer the gas dynamics in a desired way. Note that the system \eqref{eiso2aex}--\eqref{eiso2bex} can be derived from the compressible Euler system for an ideal gas by assuming constant temperature and a subsonic regime (i.e. gas velocity below the speed of sound $c$). For more details on the derivation we refer the interested reader to \cite[ISO2 model]{D-H-L-M-M-T}  and \cite[Sections 1 \& 2]{G-H-H-H}. We also point out that models of the form \eqref{iso2ex} or similar appear in a number of contributions to the literature. In subsequent sections we will however only include references that are very close to our context. For a more general reference we refer here to the monograph \cite{L}. 

Suppose that the {\it states} $p, q: (0,T) \times \Omega \to \BR$ are appropriate solutions of the system \eqref{iso2ex} and the given (non-empty, convex and closed) set $U$ encodes constraints on the controls. Then consider the rather general optimization problem 
\begin{equation}
	\begin{aligned}
	&\text{minimize}  \qquad\quad J(\Phi_1,\Phi_2,p,q) \nonumber \\
	&\text{subject to (s.t)}  \ \Phi_i \in U, i=1,2; \text{and }p,q\text{ solve } \eqref{iso2ex} ; \nonumber \\
	&\phantom{\text{subject to (s.t)}} \ k_p \leqslant p(t,x) \leqslant K_p \ \mathrm{and} \ k_q \leqslant q(t,x) \leqslant K_q\text{ for almost every }(t,x)\in(0,T)\times\Omega.
	\end{aligned}
	\tag{${P}$}\label{minexp}
\end{equation}
where $J$ is an objective functional that one wishes to minimize. To ease our exposition, in this work we take $J$ as the standard tracking-type functional which models the desire to reach a given target state over time while keeping the (average) cost of the control low. For more details see below. We also mention here that more general objectives can be used as long as they fulfill certain requirements. Further, $k_p, k_q, K_p, K_q \in \BR$ with $k_p<K_p$ and $k_q<K_q$ yield point-wise almost everywhere bound constraints on the states $p,q$.

In view of this optimization task, one of the goals of this paper is to investigate the well-posedness of \eqref{minexp} on a network of $m$ pipes, allowing non-trivial controls to act only at the \textit{boundary} of such a network. Concerning details on the latter notion we refer to the discussion in the  following section. At internal network nodes the Kirchhoff law is imposed in order to balance inflow and outflow at such nodes. Once optimal solutions of \eqref{minexp} are guaranteed, a natural next goal of this work is to derive a suitable stationarity system for characterizing such solutions via first-order conditions.

In connection with these research goals, one of the main challenges in our present study is associated with the analysis of the underlying state system which consists of nonlinear hyperbolic partial differential equations (PDEs) coupled via the Kirchhoff law on a network. Since pipes are connected by joints and in view of the dynamics of \eqref{minexp}, one needs to pay special attention to the structure of solutions and their regularity. This entails a careful choice of function spaces to enable suitable solution concepts. Another challenge stems from the non-linearity of \eqref{iso2ex} which renders the constrained problem \eqref{minexp} non-convex, even if $U$ is convex. As a consequence, both the proof of existence of optimal controls as well as their characterization via first-order systems require a careful analysis. 

Next we connect the above mathematical setting to the wider application context of implementing a transition strategy from a current (often fossil fuel prone) energy portfolio to an (ideally) all renewable one.

In fact, natural gas still plays a central role in the current European energy scenario. The so-called \emph{European Green Deal} \cite{EC2} has set a net zero greenhouse gas emissions target for the year 2050. In this context, natural gas is the key common factor in all proposed transition strategies (from fossil to renewable energy sources) \cite{EC1}. It is transported through large pipeline networks whose complexity poses a challenge to currently known modeling and analysis techniques. The intricacy of this transmission system encompasses much more than just the various hardware structures in a given network \cite{K-H-P-S} --- which are generally difficult to represent as abstract mathematical objects and concepts that can be dealt with. Apart from the fact that the partial differential equations (PDEs) governing the dynamics of gas transport are nonlinear and hyperbolic, the understanding of gas markets is constantly confronted by the rapidly evolving European Commission's gas policies. Such changes are often interpreted as (\textit{state, operator}) constraints in regards to the structure of the gas network as well as limitations imposed by legislation. The latter usually changes the gas market's dynamics by adding or removing agents (buyers, sellers, etc) or by changing their logistics. An example of such change is the replacement of the point-to-point transport routes with the entry-exit system for capacity booking and the establishment of a \textit{virtual trading point}, which allows for transactions between agents that not necessarily have direct physical access to gas volumes \cite{H-H-A}.

In a realistic gas market, one has to deal with non-cooperative agents pursuing specific objectives subject to general (global) as well as individual (private) constraints. One example for such a global constraint would be the above model of the gas flow in a pipe (or a pipeline network, more generally). The controls may then be interpreted as the agents' decision variables; in \eqref{iso2ex}, for instance, $\Phi_1$ may belong to one agent (producer) and $\Phi_2$ to another agent (wholesaler). Then, $U$ in \eqref{minexp} would encode private constraints affecting the agents.

This leads to modeling such a gas market as non-cooperative game. More precisely we will call this a generalized Nash equilibrium problem (GNEP); see, for instance, \cite{E-G-H-Z,E-G2,G-H-H-S-S-Z,H-H-A,J-vL-vW-vO}. We also refer to the seminal work \cite{N}, or \cite{A-D,F-K} for instance, for further references. Note that the descriptor ''generalized'' is used here to emphasize that each agent's optimization problem has a feasible set that depends on the decisions of its competitors.  Establishing the existence and a (first-order) characterization of associated solutions (so-called Nash equilibria) for GNEPs are typically challenged when the agents' individual problems are non-convex or when the agents' problems are posed in infinite dimensional settings. The latter is for instance the case when the above PDE model of the gas network becomes a constraint. Let us point out that GNEPs with PDE constraints are a relatively recent problem class in the literature; see e.g. \cite{D-G,G-H-S,G-H-H-S-S-Z,H-S,H-S-K,K-K-S-W,G-P-R,R,R2}. In view of our focus on energy networks we mention that the theory in \cite{H-S-K} is applied to study a GNEP for a gas market modeled by a simplified linear PDE for the gas transport. In \cite{G-H-H-S-S-Z}, GNEPs are studied in the context of gas markets where the gas transport is modeled by a linearized and viscosity regularized version of the semilinear Euler equation \eqref{iso2ex}, which is also the model of interest in this paper.

The main study object of this paper is the optimization problem which can be attributed to a single agent in the above GNEP context. It has the structure of \eqref{minexp}, but with the PDE for the gas transport posed on a tree-like network under the Kirchhoff law at interior nodes and possibly more than two controls. Let us further point out that \eqref{minexp} is also of interest independent of the GNEP resp. market context, but rather it would entail optimal boundary control of a (passive) gas network. Such a viewpoint is of interest, e.g., in steering the physical network towards a target state in terms of gas distribution.  

The two major contributions of this paper to the existing literature are the following ones:
\begin{itemize}
	\item[(i)] Based on semigroup theory, we prove the existence of smooth (local) solutions to the semilinear Euler system \eqref{iso2ex} on a tree-like network and with coupling at interior nodes (joints) via the Kirchhoff law. Technically, our approach is based on \cite{L}, but under a weak smallness assumption of the boundary data (controls); here only in the space of continuous functions, rather than in $C^1$.
	\item[(ii)] Exploiting compactness properties of the image space of the control-to-state map, i.e., the solution map for the semilinear isothermal Euler system as a mapping of the boundary controls, existence of optimal controls for \eqref{minexp} for the state system of (i) is established. Moreover, the regularity of our states allows us to derive a first-order optimality condition for characterizing a (local) solution with almost surely bounded Lagrange multipliers.
\end{itemize}

The rest of the paper is organized as follows: In Section \ref{pressec} we introduce the configuration of our state system in detail. We also fix some notation for temporal and spatial regularity spaces and comment on some of their key properties used in this work. In section \ref{main_res} we present (without proofs) and discuss our main results. In this way we hope to make the paper more accessible to the reader. The following sections are then devoted to the associated mathematical proofs. In fact, section \ref{localwell} is concerned with the well-posedness of the state system, and in Section \ref{control} we prove the existence of optimal controls as well as their first-order characterization. 

\section{The semilinear Euler system on a network of pipelines}\label{pressec}

Let $\calG = (\calE, \calV)$ be a finite, directed and connected graph whose underlying undirected graph is a tree. We denote the set of $m$ edges of $\calG$ by $\calE = \{e_1,\cdots,e_m\}$ and the set of its $n$ vertices by $\calV=\{v_1, \cdots, v_n\}.$ In this paper, $\cal{G}$ models the network of pipelines on which the analysis and the optimal boundary control of  gas transport are of interest, respectively. Each edge $e_k$ represents a pipe and each vertex $v_k$ is either a junction or a boundary node, and since $\cal{G}$ is a directed graph, we identify the \textit{start} and \textit{end} points of a given pipe $e_k$ by $0_k$ and $l_k$, respectively.

We decompose the set $\calV$ of vertices into the following three distinct subsets: entry (source or provider) vertices, denoted by $\calV_-^\partial$; exit (sink or customer) vertices, denoted by $\calV_+^\partial$; and interior (junction) nodes, denoted by $\calV^\circ$.  For the analytical description of such subsets, we first distinguish interior and boundary vertices. For this purpose, define $\xi_k: \calV \to \{-1,0, 1\}$ by
$$ \xi_k(v) := 
\begin{cases}
	-1, \ &\mbox{if} \ v = 0_k,\\
	1, \ &\mbox{if} \ v = L_k, \\
	0, \ &\mbox{otherwise}.
\end{cases}
$$
Notice that $\xi_k$ establishes a clear relationship between a given vertex $v$ and every pipe of the network. Since inner nodes are connected to more than one pipe, by introducing the set $\kappa(v) \coloneqq \{e_k \in \calE: \xi_k(v) \neq 0\}$ we can identify inner and boundary vertices via
\begin{equation}\label{innerbound}
	\calV^\circ = \{v \in \calV; |\kappa(v)|>1\}, \qquad \calV^\partial = \calV \setminus {\calV^\circ}.
\end{equation} 
Entry and exit vertices can also be characterized by $\xi_k$. Indeed, we have 
\begin{equation}\label{entryexit}
	\calV_\pm^\partial \coloneqq \{v \in \calV^\partial: \xi_k(v) = \pm 1 \ \mbox{for some } k\}.
\end{equation} 
An example of a network as described above is shown in Figure \ref{net1}. Also, we refer to \cite{G-H2} for a similar description of such networks. 
\begin{figure}[t]
	\centering
	\begin{tikzpicture}[scale=.5, transform shape]
		\node [circle, draw] (a) at (0, 1) {$v_2$} ;
		\node [circle, draw,color=blue] (g) at (1, 5) {$v_3$} ;
		\node [circle, draw,color=blue] (b) at (4, -1) {$v_1$} ;
		\node [circle, draw] (c) at (4, 3) {$v_4$} ;
		\node [circle, draw] (d) at (8, 1) {$v_5$} ;
		\node [circle, draw,color=red] (e) at (8, 5) {$v_6$} ;
		\node [circle, draw,color=red] (f) at (13, 0) {$v_7$} ;
		\draw [->,line width=1] (g) -- (c) node[pos=.5,above] {$e_3$} ;
		\draw [->,line width=1] (b) -- (a) node[pos=.5,above] {$e_1$} ;
		\draw [->,line width=1] (a) -- (c) node[pos=.5,above] {$e_2$};
		\draw [->,line width=1] (c) -- (d) node[pos=.5,above] {$e_4$};
		\draw [->,line width=1] (c) -- (e) node[pos=.5,above] {$e_5$};
		\draw [->,line width=1] (d) -- (f) node[pos=.5,above] {$e_6$};
	\end{tikzpicture}
	\caption{\footnotesize Example of a network with $m = 6$ and $n = 7$. Notice that in this case $\calV^\circ = \{v_2, v_4, v_5\}$, $\calV_-^{\partial} = \{v_1, v_3\}$ and $\calV_=^\partial = \{v_6,v_7\}.$}\label{net1}
\end{figure}
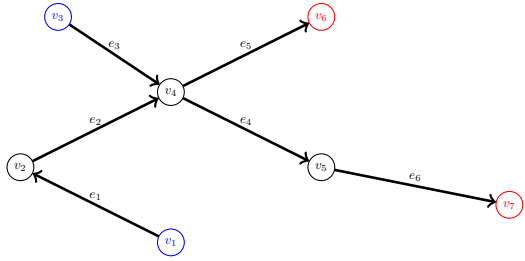 

\

A pipe $e_k$ in $\calG$ is assumed to have length $L_k > 0$, and it is considered to be cylindrical with a circular cross section of diameter $D_k>0$. We also consider its associated friction coefficient $\lambda_k \geqslant 0$ and its inclination $\sin(\alpha_k), \alpha_k \in (-\pi/2, \pi/2)$. It is typical to assume $L_k \gg D_k$ such that the gas dynamics are well described by one dimensional models \cite{D-H-L-M-M-T}.  For each time instance $t \geqslant 0$, the gas pressure and flux are functions from the finite interval $[0,L_k]$ to $\mathbb{R}$, respectively. As motivated in the introduction, the evolution of these distributions is assumed to obey the isothermal semilinear Euler system. We restate this system here, but now with a focus on introducing suitable initial and boundary data (for pipes connected to entry or exit vertices) and continuity conditions (for junctions). 

For a fixed time horizon $T>0$, define the open interval $\Omega_k \coloneqq (0,L_k)$ with product $\Omega \coloneqq \prod_{k=1}^m \Omega_k$, and let $Q_T^k, Q_T$ denote the cylinders $Q_T^k \coloneqq (0,T) \times \Omega_k$ and $Q_T \coloneqq (0,T) \times \Omega$. On the $k^{\tiny\mbox{th}}$-pipe, the Euler system relates the quantities, pressure $p^k = p^k(t,x)$ and flux $q^k = q^k(t,x)$ 
via the equations 
\begin{subnumcases}{\label{ksys}}
	\dfrac{\partial p^k}{\partial t} + c^2\dfrac{\partial q^k}{\partial x} = 0 & a.e. in $Q_T^k$, \label{iso2na} \\[2mm]
	\dfrac{\partial q^k }{\partial t} + \dfrac{\partial p^k}{\partial x} + \gamma_k p^k= -\beta_k\dfrac{q^k|q^k|}{p^k}  & a.e. in $Q_T^k$,\label{iso2nb} 
\end{subnumcases} 
with
\begin{equation}
	\label{betagamma} \beta_k \coloneqq \dfrac{\lambda_k}{2D_k}, \qquad \gamma_k \coloneqq \dfrac{g\sin(\alpha_k)}{c^2},
\end{equation} 
where $c$ denotes the sound speed, and $g$ is the acceleration of gravity. Here and below, 'a.e.' stands for 'almost everywhere' in the sense of the Lebesgue measure.

We shall assume that initially at time $t = 0$ the pressure and flux distributions in each pipe are given by an equilibrium or steady state (SS) solution of the corresponding counterpart of the Euler system. This assumption is standard in the context of boundary control of compressible Euler equations; see, for instance, \cite{G-H} and \cite{H-S2}. To compute the SS solution, we assume that the pressure is known at the entry vertex of each pipe and collect this information in the vector $\B{p}_{\mathrm{e}}^{\mathrm{in}} = (p_{\mathrm{\mathrm{in}}}^k)\kk{k}$. The constant SS mass flux vector is denoted by $\B{q}_{\mathrm{e}} = (q_{\mathrm{e}}^k)\kk{k}$ and the SS pressure $\B{p}_{\mathrm{e}} = (p_{\mathrm{e}}^k)\kk{k}$ is given, in each pipe, by 
$$p_{\mathrm{e}}^k(x) = \sqrt{ e^{-2\gamma_k x}\left[(p_{\mathrm{\mathrm{in}}}^k)^2 -\dfrac{\beta_k (q_{\mathrm{e}}^k)^2}{\gamma_k} (e^{2\gamma_k x}-1)\right]}, \qquad x \in [0,L_k).$$
 We define $p_{\mathrm{\mathrm{out}}}^k \coloneqq p_{\mathrm{e}}^k(L_k), \B{p}_{\mathrm{e}}^{\mathrm{out}} = (p_{\mathrm{\mathrm{out}}}^k)\kk{k}.$ Finally, the SS solution on each pipe is denoted by $\vec{v}_{\mathrm{e}}^k(\cdot) \coloneqq (p_{\mathrm{e}}^k(\cdot), q_{\mathrm{e}}^k(\cdot))$, and on the network by $\B{v}_{\mathrm{e}}: \Omega \to \BR^{2m}$, $\B{v}_{\mathrm{e}} := (\vec{v}_{\mathrm{e}}^k)\kk{k}.$ We shall assume that $\B{p}_{\mathrm{e}}^{\mathrm{in}}$ and $\B{q}_{\mathrm{e}}$ are such that each component of the SS solution is continuously differentiable on its respective domain and $p_{\mathrm{e}}^k$ is monotonically decreasing for all $k$ with $p_{\mathrm{e}}^k(x) \geqslant p_{\mathrm{\mathrm{out}}}^k$ for all $x \in \overline{\Omega_k}=[0,L_k]$. Moreover, we assume $\B{v}_{\mathrm{e}}$ to be compatible with the pressure and flux continuity conditions that we wish solutions to enjoy. This will be further discussed below.

The condition on the \textit{continuity of the flux} serves the fact that at each inner node of the network the amount of gas that streams in needs to flow out again. We model this by imposing the so-called \textit{Kirchhoff} condition at each node. In our context the latter reads
\begin{equation}\label{kirchoff} 
	\sum\limits_{k = 1}^m \xi_k(v)D_k^2q^k(v,t) = 0, \qquad v \in \calV^\circ, \ t \in [0,T).
\end{equation}

Concerning the \textit{continuity of the pressure} we require the pressure to be stable at junctions, i.e., equal in all of the pipes meeting at a junction. Mathematically, we require that for any $1 \leqslant j, k \leqslant m$ and $v \in \calV^\circ$ such that both $\xi_k(v), \xi_j(v)$ are nonzero 
\begin{equation}
	\label{pressurecont} p^j(v,t) = p^k(v,t),\qquad\text{for all }t \in [0,T).
\end{equation} 

The main goal of this paper is to study how controls acting on the boundary of the network can drive the SS solution (assumed to be the state at $t=0$) to a given target distribution while satisfying fixed state constraints. We are going to assume that the pressure is controlled at the set $\calV_-^\partial$ of entry vertices and the flux is controlled at the set $\calV_+^\partial$ of exit vertices. That is, $p^k(\cdot,0)$ (if $0_k \in \calV_-^\partial$) and $q^k(\cdot,L_k)$ (if $l_k \in \calV_+^\partial$) are given functions on $(0,T)$ of appropriate regularity. Again, this structure (pressure controlled at the entry and flux at the exit) is not new in the context of the (optimal) control of gas transport; see \cite{D-L}, for instance.

{\bf Notation.} We now introduce the notation which we are going to use throughout the paper. 

For an open set $O \subset \BR$ and $1 \leqslant p < \infty$, $L^p(O)$ denotes the space of Lebesgue measurable functions whose absolute value to the $p^{\tiny \mbox{th}}$-th power is Lebesgue integrable. For $p = \infty$, the set $L^\infty(\Omega)$ denotes the set of Lebesgue measurable functions which are also almost everywhere bounded in $O$. When equipped with the norm 
\begin{equation*} 
	\|u\|_{L^p(O)} \coloneqq 
	\begin{cases} 
		\left(\displaystyle\int_O |p(x)|^pdx\right)^{1/p} & \ \mbox{for} \ 1\leqslant p < \infty, \\[2mm] \mbox{ess}\sup\limits_{x \in O} |u(x)|, & \ \mbox{for} \ p = \infty, 
	\end{cases}
\end{equation*} 
$L^p(O)$ is a Banach space. In the case $p=2$, $L^2(O)$ is a Hilbert space if equipped with the standard inner product $$(u,v)_{L^2(O)} \coloneqq \int_O u(x)v(x)dx$$. The latter induces the norm $\|\cdot\|_{L^2(O)}$. For a non-negative integer $s$ we denote by $W^{s,p}(O)$ the Sobolev space of $L^p(O)$ functions whose distributional derivatives up to order $s$ are also in $L^p(O)$. With the norm
\begin{equation*}
	\|u\|_{W^{s,p}(O)} \coloneqq 
	\begin{cases} \left(\sum\limits_{\alpha \leqslant s}\|D^\alpha u\|_{L^p(O)}^p\right)^{1/p} & \ \mbox{for} \ 1\leqslant p < \infty, \\[2mm] \sum\limits_{\alpha \leqslant s} \|D^\alpha u\|_{L^\infty(O)}, & \ \mbox{for} \ p = \infty, \end{cases}
\end{equation*} 
$W^{s,p}(O)$ (for $s \in \mathbb{N} \cup \{0\}$, $1 \leqslant p \leqslant \infty$) is a Banach space. In the case $p = 2$, the associated spaces are also Hilbert spaces. It is then common to write $H^s(O)$ instead of $W^{s,2}(O)$. The inner product $$(u,v)_{H^s(O)} = \sum\limits_{\alpha \leqslant s} (D^\alpha u, D^\alpha v)_{L^2(O)}$$ induces the norm $\|\cdot\|_{H^s(O)}$. 

We denote by $C(\overline O)$ the set of functions $u: O \to \BR$ which are continuous on $O$ and can be extended continuously to $\partial O$, the boundary of $O$. In the particular case where $O = (a,b)$, with $a,b\in\mathbb{R}$ and $a<b$, we denote by $\calM(O)$ the dual of $C(\overline O)$. It is well known that $\calM(O)$ can be identified with the set of regular and finite Borel measures. More specifically (since we are in dimension one), any $l \in \calM(O)$ can be uniquely (up to a countable set) represented by a function $\mu$ of bounded variation and such that $\mu(a) = 0$ and $$l(f) = \int_a^b fd\mu$$ for all $f \in C(\overline O).$ This justifies the duality pairing between $\calM(O)$ and $C(\overline O)$ to be defined as $$\langle \mu, f \rangle_{\calM(O),C(\overline O)} = \int_O fd\mu.$$

 If $I \subset \BR$ is an interval and $X$ a Banach space, then we denote by $C^m(I;X)$ the set of $X$-valued functions $u: I \subset \BR \to X$ which are $m$-times continuously differentiable on $I$. When $I$ is compact, then the space $C^m(I;X)$ when equipped with the norm $$\|u\|_{C^m(I;X)} \coloneqq \sup\limits_{t \in I}\sum\limits_{k=0}^m \|\partial_t^k u(t)\|_X$$ is a Banach space.  

For $1 \leqslant p \leqslant \infty$ we denote by $L^p(I;X)$ the Bochner space of functions $u: I \subset \BR \to X$ such that $t \mapsto \|u(t)\|_X$ belongs to $L^p(I).$ Equipped with the norm 
$$\|u\|_{L^p(I;X)} \coloneqq 
\begin{cases} 
	\left(\displaystyle\int_I \|u(t)\|^pdt\right)^{1/p} & \ \mbox{for} \ 1\leqslant p < \infty, \\[2mm] \mbox{ess}\sup\limits_{t \in I} \|u(t)\|_X, & \ \mbox{for} \ p = \infty, 
\end{cases}$$ 
$L^p(I;X)$ is a Banach space, and if $X'$ denotes the topological dual of $X$ then $\left[L^p(I;X)\right]' = L^q(I;X')$ where $pq = p+q$ for $1<p,q<\infty$. If $X$ is a Hilbert space then $L^2(I;X)$ is a Hilbert space when equipped with the inner product $$(u,v)_{L^2(I;X)} = \int_I (u(t),v(t))_{X}dt.$$ Similarly, $W^{s,p}(I,X)$ denotes the space of functions in $L^p(I,X)$ whose distributional derivatives up to order $s$ are also in $L^p(I,X).$ With the norm 
\begin{equation*}
	\|u\|_{W^{s,p}(I,X)} \coloneqq 
	\begin{cases} 
		\left(\sum\limits_{\alpha \leqslant s}\|D^\alpha u\|_{L^p(I,X)}^p\right)^{1/p} & \ \mbox{for} \ 1\leqslant p < \infty, \\[2mm] \sum\limits_{\alpha \leqslant s} \|D^\alpha u\|_{L^\infty(I,X)}, & \ \mbox{for} \ p = \infty, 
	\end{cases}
\end{equation*}
When $p = 2$ and $X$ is a Hilbert space, the Hilbert space $W^{s,2}(I,X)$ is denoted by $H^s(I,X)$ and the inner product is given by  $$(u,v)_{H^s(I,X)} = \sum\limits_{\alpha \leqslant s} (D^\alpha u, D^\alpha v)_{L^2(I,X)}.$$

Since, in our application context, on each pipe the description of the evolution is given by two functions --- pressure and flux --- and we have $m$ pipes, we are led to work with Cartesian products of Lebesgue and Sobolev spaces. To ease notation, we use \emph{bold}  typeface ($\BBL, \BBH, $ etc) to denote the products on each pipe and \emph{double-struck} typeface ($\BL, \BH, $ etc) when extending it to the network. 

For the sake of convenient reference let us state the following definition where we write $x_k\in (0,L_k)=\Omega_k$ in order to refer to the space variable along the $k$th pipe.
\begin{definition}\label{solution} 
	We say that the vector-valued function $\B{v}: (0,T) \times \Omega \to \BR^{2m}$ is a solution of the isothermal semilinear Euler system if for each $(t,\B{x}) \in (0,T) \times \Omega$, $\B{v}$ has the form $\B{v}(t,\B{x}) = (\vec{v}^k(t,x_k))\kk{k}$, $\B{x} = (x_k)\kk{k}$ with $\vec{v}^k = (p^k, q^k)$ and 
	\begin{itemize}
		\item[\bf (i)] each $\vec{v}^k$ solves the corresponding system \eqref{ksys};
		\item[\bf (ii)] the initial, boundary and continuity conditions are satisfied.
	\end{itemize}  
\end{definition} 

Later in the text we will also specify in which sense $\vec{v}^k$ is required to solve the PDE on the $k$th pipe. For the time being, we are merely interested in structural aspects of our setting. Below we will also use $\BBL_k^2(\Omega_k) \coloneqq L^2(\Omega_k) \times L^2(\Omega_k)$ with the norm $$\|\vec{w}\|_{\BBL_k^2(\Omega_k)}^2 := D_k^2\left(\|w_1\|_{L^2(\Omega_k)}^2 + c^2\|w_2\|_{L^2(\Omega_k)}^2\right),$$ and $\BBH_k^1(\Omega_k) \coloneqq H^1(\Omega_k) \times H^1(\Omega_k)$ equipped with the norm $$\|\vec{w}\|_{\BBH_k^1(\Omega_k)}^2 \coloneqq \|\vec{w}\|_{\BBL_k^2(\Omega_k)}^2 + \|\nabla \vec{w}\|_{\BBL_k^2(\Omega_k)}^2.$$ On the network, we let $\BL^2(\Omega) \coloneqq 
\prod_{k=1}^m \BBL_k^2(\Omega_k)$ 
and $\BH^1(\Omega) \coloneqq 
\prod_{k=1}^m \BBH_k^1(\Omega_k)$ 
both endowed with the induced Euclidean norm. 

In view of regularity in time, given two Banach spaces $X, Y$, with $X$ continuously embedded in $Y$, i.e., $X \hookrightarrow Y$ we define
\begin{equation*}
	\BBX_{T}(X,Y) \coloneqq  C([0,T];X) \cap C^1([0,T], Y),
\end{equation*} with the norm $\|\vec{w}\|_{\BBX_T} \coloneqq \sup_{t \in [0,T]} \left(\|\vec{w}(t,\cdot)\|_{X} + \|\vec{w}_t(t,\cdot)\|_{Y}\right)$,
and for $1 \leqslant p \leqslant \infty$ 
\begin{equation}\label{YXY}
	\BBY_{T}^p(X,Y) \coloneqq  L^p(0,T;X) \cap H^1(0,T; Y),
\end{equation} 
with the norm $\|\vec{w}\|_{\BBY_T^p} \coloneqq \|\vec{w}(\cdot,t)\|_{L^p(0,T;X)} + \|\vec{w}_t(\cdot,t)\|_{L^2(0,T;Y)}$. 

In this paper we will often use the spaces defined above with $X = \BH^1(\Omega)$ and $Y = \BL^2(\Omega)$, or spaces of similar structure. A key observation related to this choice is that, due to the Aubin--Lions lemma \cite[Corollary 4, p. 84]{S} the space $\BBY_T^p \ (p < \infty)$ (resp. $\BBY_T^\infty$) lies in between two spaces contained in $L^p\left(0,T;[C(\overline\Omega)]^2\right)$ (resp. $C\left([0,T] \times \overline{\Omega}^2\right)$). Namely, we have 
\begin{equation}\label{embed2}
	\BBX_T(\BH^1(\Omega),\BL^2(\Omega)) \hookrightarrow \BBY_T^p(\BH^1(\Omega),\BL^2(\Omega)) \hookrightarrow L^p\left(0,T;[C(\overline\Omega)]^2\right),
\end{equation}
with the second embedding being compact and 
\begin{equation}\label{embed}
	\BBX_T(\BH^1(\Omega),\BL^2(\Omega)) \hookrightarrow \BBY_T^\infty(\BH^1(\Omega),\BL^2(\Omega)) \hookrightarrow C\left([0,T] \times \overline{\Omega}^2\right),
\end{equation} 
with the second embedding being again compact. 
We finish this section by defining yet another important space for us used when showing the well-posedness of the semilinear Euler system. Given a compact set $\mathcal{K} \subset \BR^{2m}$, we define for some fixed $\kappa_t,\kappa_\B{x}>0$ 
\begin{equation}\label{compactsmooth}
\BBX_T^\mathcal{K}(\BH^1(\Omega),\BL^2(\Omega)) \coloneqq \left\{\B{v} \in \BBX_T(\BH^1(\Omega),\BL^2(\Omega)): 
\begin{matrix}
	\quad\:\B{v}(t,\B{x}) \in \mathcal{K} \ \mbox{for all} \ (t,\B{x}) \in \overline{Q}_T, \\
	\|D_t\B{v}\|_{C([0,T];\BL^2(\Omega))} \leqslant \kappa_t, \\
	\|D_\B{x}\B{v}\|_{C([0,T];\BL^2(\Omega))} \leqslant \kappa_\B{x}
\end{matrix}\right\}.
\end{equation}   

\section{Main results and discussion}\label{main_res}

As indicated in \eqref{minexp}, in our main optimal control problem we work with pointwise a.e \textit{box-} constraints on the state. In this paper, we use $\BCK\subset \BR^{2m}$ to indicate this particular compact set.

With $$\BV_T \coloneqq C\left([0,T]; \prod\limits_{k = 1}^m C\left( \overline{\Omega_k} \right) \times C\left( \overline{\Omega_k}\right)\right),$$ and given a compact $\mathcal{K} \subset \BR^{2m}$ we introduce the $\mathcal{K}$--admissible set
\begin{equation}\label{stateconstraint}
	\BV_{T,\BCK}^\mathrm{ad} \coloneqq \{\B{v} \in \BV_T : \B{v}(t,\B{x}) \in {\mathcal{K}} \ \mbox{for all} \ (t,\B{x}) \in \overline{Q}_T\}.
\end{equation} 

Due to the structure of the nonlinearity in the Euler system, the compact set $\BCK$ cannot be arbitrary. This requires us to introduce the notion of \emph{suitable} compact sets.
\begin{definition}\label{suitable}
	Let ${\mathcal{K}} \subset \BR^{2m}$ be a compact set. We say that ${\mathcal{K}}$ is {\bf suitable} for state constraint representation if there exists an element $\B{v}_\mathrm{e}(\B{x}) \in \mathrm{int}(\BCK)$ for all $\B{x} \in \overline\Omega$ and $${\mathcal{K}} \coloneqq \bigcup\limits_{k=1}^m {\mathcal{K}}_k$$ where ${\mathcal{K}}_k \coloneqq {\mathcal{K}}_k^\mathrm{p} \times {\mathcal{K}}_k^\mathrm{q}$ with ${\mathcal{K}}_k^\mathrm{p} \coloneqq [a_k, b_k] \subseteq \BR_+^\ast$, ${\mathcal{K}}_k^\mathrm{q} \coloneqq [c_k,d_k] \subseteq \BR$ and, for each $k$, $a_k + p_{\mathrm{in}}^k \leqslant b_k.$
\end{definition} 

Throughout we invoke the following assumption concerning the initial state.
\begin{assumption}\label{SSnet}
	The vectors $\B{q}_{\mathrm{e}}$ and $\B{p}_{\mathrm{e}}^{in}$ are taken such that $q_{\mathrm{e}}^k, p_{\mathrm{in}}^k,p_{\mathrm{out}}^k> 0$ for all $1 \leqslant k \leqslant m$ and the continuity conditions \eqref{kirchoff} and \eqref{pressurecont} are both satisfied.
\end{assumption} 
We notice that Assumption \ref{SSnet} above is sufficient to guarantee that $\B{v}_{\mathrm{e}}$ is continuously differentiable and each component is, in addition, strictly monotonically decreasing. However, to comply with the state constraints, we also needed to add the interiority condition on $\B{v}_\mathrm{e}$.

Let $\BU \subset [H^2(0,T)]^{2m}$ be defined by
\begin{equation}\label{defU} 
	\BU \coloneqq \left\{\Phi \in [H^2(0,T)]^{2m}: 
	\begin{matrix}
		\Phi_{2k+1} \equiv 0 \ \mathrm{if} \ 0_k \notin \calV_-^\partial, \ 0 \leqslant k \leqslant m-1, \\
		\Phi_{2k}
		\equiv 0 \ \mathrm{if} \ l_k \notin \calV_+^\partial, \ 1 \leqslant k \leqslant m
	\end{matrix}\right\},
\end{equation}  
and let $\BU_0$ be defined as 
\begin{equation}
	\label{defU0} \BU_0 \coloneqq \{\Phi \in \BU: \Phi(0) = 0\}.
\end{equation}
Notice that $\BU_0$ is a subspace of $\BU$ and both are Hilbert spaces when endowed with the inner product of $[H^2(0,T)]^{2m}.$ We now define the set of admissible controls.
\begin{definition}\label{admbc2} 
	Let $\Phi^{\mathrm{e}} = (\Phi_k^{\mathrm{e}})_{k=1}^{2m}$ be defined by
	\begin{equation}\label{phi2}
		\begin{cases}
			\Phi_{2k+1}^{\mathrm{e}}\equiv  
			\begin{cases}
				p_{\mathrm{in}}^{k+1} \ &\mathrm{if} \ 0_k \in \calV_-^\partial, \\
				0 \ &\mathrm{otherwise,}
			\end{cases} \ &\mathrm{for} \ 0 \leqslant k \leqslant m-1, \\
			\Phi_{2k}^{\mathrm{e}} \quad\equiv 
			\begin{cases}
				q_{\mathrm{e}}^k \ \quad&\mathrm{if} \ l_k \in \calV_+^\partial, \\
				0 \ &\mathrm{otherwise,}
			\end{cases} \ &\mathrm{for} \ 1 \leqslant k \leqslant m.
		\end{cases}
	\end{equation}
	We say that $\Phi \in \BU$ belongs to $\BU^\mathrm{ad}$, the set of admissible controls, if 
	\begin{itemize} 
		\item[\bf(i)] $\Phi-\Phi^{\mathrm{e}} \in \BU_0,$ i.e., $\Phi(0) = \Phi^{\mathrm{e}}$, and $$\displaystyle\sum_{k=1}^{2m}\left\|\Phi_k-\Phi_k^{\mathrm{e}}\right\|_{H^2(0,T)}^2 \leqslant \eta^2$$ for a fixed $\eta>0.$
		
		\item[\bf (ii)] $\|\Phi-\Phi^\mathrm{e}\|_{[C([0,T])]^{2m}} \leqslant \kappa_\BU $ with fixed $\kappa_\BU>0$ sufficiently small.
	\end{itemize}  
\end{definition}
The smallness condition on $\kappa_\BU$ in (ii) above will be specified later.
\begin{remark}\label{closedconvexbounded}
	Notice that $\BU^\mathrm{ad}$ is a bounded, closed and convex subset of $[H^2(0,T)]^{2m}$
\end{remark} 
\begin{remark}
	In this paper we work with a so called \emph{passive} network. This means that non-zero controls act only on its boundary. The structure of the set $\BU$ in \eqref{defU} allows for the incorporation of non-passive elements like valves and compressor stations (see \cite{G-H}) -- at least from the abstract point of view -- although their modeling would probably require smoothing techniques in order to impose pointwise state constraints.
\end{remark}
\begin{remark}\label{bcreg} 
	In Section \ref{pressec} we said that the boundary data were given functions of appropriate regularity. For studying the general well-posedness problem as we do in Section \ref{localwell}, boundary conditions are assumed to be fixed. However, when we study the control problem, those become the variable of interest. In any case, by \emph{appropriate} regularity we mean that $\Phi = (\Phi_k)_{k =1}^{2m}$ is defined as 
	\begin{equation}\label{phi4} 
		\begin{cases}
			\Phi_{2k+1} =  
			\begin{cases}
				p^{k+1}(0,\cdot) \ &\mathrm{if} \ 0_k \in \calV_-^\partial, \\
				0 \ &\mathrm{otherwise}
			\end{cases} \ &\mathrm{for} \ 0 \leqslant k \leqslant m-1, \\
			\Phi_{2k} \quad = 
			\begin{cases}
				q^k(L_k,\cdot) \ \quad &\mathrm{if} \ l_k \in \calV_+^\partial, \\
				0 \ &\mathrm{otherwise}
			\end{cases} \ &\mathrm{for} \ 1 \leqslant k \leqslant m,
		\end{cases}
	\end{equation} 
	and belongs to $\BU^\mathrm{ad}.$
\end{remark}
\begin{remark}\label{U0eta}
	Let 
	\begin{equation}
		\BU_0^\eta \coloneqq \left\{\Phi \in \BU_0: \ \|\Phi\|_\BU \leqslant \eta, \ \|\Phi\|_{[C([0,T])]^{2m}} \leqslant \kappa_\BU\right\}
	\end{equation} 
	and notice that $\Phi + \Phi^\mathrm{e} \in \BU^\mathrm{ad}$ for all $\Phi \in  {\BU_0^\eta}$. Moreover, by defining $\phi: {\BU_0^\eta} \to \BU^{\mathrm{ad}}$ as $\phi(\Phi) = \Phi + \Phi^\mathrm{e}$ we see that ${\BU_0^\eta}$ and $ \BU^{\mathrm{ad}}$ are isomorphic.
\end{remark}
Now that all the relevant spaces have been introduced, we present a Sobolev embedding and interpolation theorem in one dimension. The version presented here is a combination of the Theorems 2.5.4 and 2.6.4 in \cite[p. 90 and p. 94, respectively]{K} and \cite[p.19]{L-M}, adapted to our setting.
\begin{proposition}
	For $\Omega := \prod\limits_{k=1}^{m} \Omega_k \times \Omega_k$ and $T>0$ we have that
	\begin{itemize} 
		\item[\bf (i)] the embedding $\BH^1(\Omega) \hookrightarrow \BL^p(\Omega)$ $(1 \leqslant p \leqslant \infty)$ is continuous and the inequality
		\begin{equation}\label{emb_1p}
			\|\B{u}\|_{\BL^p(\Omega)} \leqslant \kappa_{1,p}\|\B{u}\|_{\BH^1(\Omega)},
		\end{equation}
		holds for all $\B{u} \in \BH^1(\Omega).$ 
		\item[\bf (i)] the embedding $\BH^1(\Omega) \hookrightarrow C(\overline\Omega)$ is compact and the inequality
		\begin{equation}\label{emb_1inf}
			\|\B{u}\|_{C(\overline\Omega)} \leqslant \kappa_{1,c}\|\B{u}\|_{\BH^1(\Omega)},
		\end{equation}
		holds for all $\B{u} \in \BH^1(\Omega).$ 
		\item[\bf (iii)] for any $0<\varepsilon \ll 1/2,$ the embedding $H^{1/2+\varepsilon}(\Omega) \hookrightarrow C(\overline\Omega)$ is continuous and the interpolation inequalities
		\begin{equation}\label{emb_epsinf}
			\|\B{u}\|_{C(\overline\Omega)} \leqslant \kappa_{\varepsilon,c} \|\B{u}\|_{\BH^{1/2+\varepsilon}(\Omega)},
		\end{equation}
		\begin{equation}\label{emb_epsinter}
			\|\B{u}\|_{\BH^{1/2+\varepsilon}(\Omega)} \leqslant \kappa_{\varepsilon} \|\B{u}\|_{\BH^1(\Omega)}^{1/2 + \varepsilon}\|\B{u}\|_{\BL^2(\Omega)}^{1/2 - \varepsilon}
		\end{equation}
		hold for all $\B{u} \in \BH^{1/2+\varepsilon}(\Omega)$ and $\B{u} \in \BH^1(\Omega)$, respectively.
		\item[\bf (iii)] the embedding $\BU \hookrightarrow [C^1([0,T])]^{2m}$ is continuous and the inequality  \begin{equation}\label{emb_12}
			\|\Phi\|_{[C^1([0,T])]^{2m}} \leqslant \kappa_{1,\BU}\|\Phi\|_{\BU},
		\end{equation}
		holds for all $\Phi \in \BU.$ 
	\end{itemize}
\end{proposition}
It should be noted that we introduced the above with distinct indices. But particularly in Section \ref{localwell}, the various embeddings will be used without further explanation or possibly any reference to $m$.

In order to formulate our main results, we reintroduce our state equations in a unified abstract setting.  For this we first formally introduce several abstract operators and some notation. Let $\BBB_0^k$ and $\BBB_1^k(x)$ act, formally, on a vector $\vec{w} = (w_1, w_2)$ according to
\begin{equation}\label{bdop} 
	\BBB_0^k\vec{w} = 
	\begin{bmatrix}
		0 & -\dfrac{c^2}{L_k} \\ \dfrac{1}{L_k} & 0
	\end{bmatrix} \vec{w},\qquad \BBB_1^k(x)\vec{w} \coloneqq 
	\begin{bmatrix}
		\dfrac{L_k-x}{L_k}w_1 \\[4mm] \dfrac{x}{L_k}w_2
	\end{bmatrix} = 
	\begin{bmatrix}
		\dfrac{L_k-x}{L_k} & 0 \\ 0 & \dfrac{x}{L_k}
	\end{bmatrix}\vec{w},
\end{equation} 
and define 
\begin{equation}\label{defB}
	\BB_0 \coloneqq \mbox{diag}\{\BBB_0^1, \cdots, \BBB_0^m\}, \qquad \BB_1 \coloneqq \mbox{diag}\{\BBB_1^1, \cdots, \BBB_1^m\}.
\end{equation} 
For the rest of the paper, we will not carry on the dependence of $\BB_1^k$ on $\B{x} \in \BR^{2m}$. Next, recalling that $\gamma_k = \dfrac{g}{c^2}\sin(\alpha_k)$ we introduce operators $\BBA_k$ and $\BBP_k$ as 
\begin{equation}\label{aandf}
	\BBA_k\coloneqq
	\begin{bmatrix}
		0 & -c^2\partial_x \\ -\partial_x & 0
	\end{bmatrix},\qquad \BBP_k\coloneqq 
	\begin{bmatrix}
		0 & 0 \\ -\gamma_k & 0
	\end{bmatrix}
\end{equation}
and
\begin{equation}\label{defAP} 
	\BA \coloneqq \mbox{diag}\left\{\BBA_1, \cdots, \BBA_m\right\} \qquad \BP \coloneqq \mbox{diag}\left\{\BBP_1, \cdots, \BBP_m\right\}, 
\end{equation} 
with domains to be appropriately defined after equation \eqref{abs1} below.
Finally, with $\vec{v}^k \coloneqq (p^k,q^k)$, $\B{v} \coloneqq (\vec{v}^k)\kk{k}$ and recalling that  $\beta_k \coloneqq \dfrac{\lambda_k}{2D_k}$, we define
\begin{equation}\label{begF}
	\BBF^k(\B{v}) \coloneqq \left(0,-\beta_k\dfrac{q^k|q^k|}{p^k}\right)
\end{equation}
and 
$\BF(\B{v})\coloneqq (\BBF^k(\B{v}))\kk{k}.$ 
Then, for a (network)--boundary datum $\Phi \in \BU^\mathrm{ad}$ represented as in \eqref{phi4}, the semilinear Euler system \eqref{ksys} can be rewritten as 
\begin{subnumcases}{\label{abs1}}
	\B{u}_t = (\BA+\BP) \B{u} + \BB\Phi + \BF(\B{v}), \label{e232nanew} \\[2mm]
	\B{u}(0) =\B{u}_0, \label{e232nbnew} 
\end{subnumcases} 
where $\B{u} \coloneqq \B{v} - \BB_1\Phi$,  $\BB \coloneqq \BB_0 + \BP\BB_1 - D_t\BB_1$, the domain of $\BP$ is $\calD(\BP) = \BL^2(\Omega)$ and $\calD(\BA)$ is characterized by: $\B{u} \in \calD(\BA)$ if and only if 
\begin{itemize} 
	\item[\bf(i)] $\B{u} = (\vec{u}_1, \cdots, \vec{u}_m) \in \BH^1(\Omega)$, $\vec{u}_k = (\tilde p^k, \tilde q^k)$; 
	\item[\bf(ii)] $\tilde p^k(0_k) = 0$ for all $k$ such that $0_k \in \calV_-^\partial$; 
	\item[\bf(iii)]$\tilde q^k(L_k) = 0$ for all $k$ such that $l_k \in \calV_+^\partial$; and
	\item[\bf(iv)] both continuity conditions \eqref{pressurecont} and \eqref{kirchoff} are satisfied at all the inner nodes.
\end{itemize}
Let us next introduce the notion of a classical (smooth) solution.
\begin{definition}[\bf Classical solution]\label{cl_sol}
	Let $\Phi \in \BU^\mathrm{ad}$. We say that $\B{v}: \Omega \times (0,T) \to \BR^{2m}$ is a {\bf classical solution} of the semilinear Euler system with boudary condition given by $\Phi$ provided $\B{v} \in \BBX_{T}(\BH^1(\Omega),\BL^2(\Omega))$, $\B{v}(0) = \B{v}_{\mathrm{e}}$, is such that $\B{v}_k = (p^k,q^k)$ solves the corresponding semilinear Euler system \eqref{ksys} and satisfies both, the boundary and the continuity conditions.
\end{definition}
Notice that initial, boundary \emph{and} continuity conditions have to be accounted for separately in the above definition. In this sense, the introduction of the abstract system \eqref{abs1} helps to simplify the analysis, because now both, the boundary and continuity conditions are embedded into the definition of the differential operator $\BA$. The main property of $\BA$ allowing us to even consider the abstract version of the system is state below and discussed in detail in Section \ref{localwell}. 
\begin{theorem}\label{gen}
	The operator $\BA: \calD(\BA) \subset \BL^2(\Omega) \to \BL^2(\Omega)$ is {\bf skew--adjoint}. Therefore, it generates a $C_0$--{\bf group} of isometries.
\end{theorem}
Of course, for the purpose of this study, the generation of a strongly continuous semigroup is enough. However, we chose to stated Theorem \eqref{gen} in full strength since such properties come naturally in the proof of the generation results. Now, since $\BA$ is the generator of a semigroup, the notion of classical solutions is borrowed from semigroup theory.
\begin{lemma}\label{ntoabs}
	Let $\Phi \in \BU^\mathrm{ad}$. Then $\B{v}:  (0,T) \times \Omega \to \BR^{2m}$ is a {\bf classical solution} of the semilinear Euler system with boundary condition given by $\Phi$ if and only if $\B{u} = \B{v} - \BB_1\Phi \in \BBX_{T}(\calD(\BA),\BL^2(\Omega))$ solves \eqref{abs1} in the classical sense.
\end{lemma}
We work with the abstract formulation of the state system from now on. Based on this, we want to study the well-posedness of the following minimization problem:
\begin{equation}
	\begin{aligned}
		\min\limits_{(\Phi,\B{v}) \in \BU \times \BV} \ &\hat J(\Phi,\B{v}) \coloneqq \dfrac{1}{2}\|\B{v}-\B{v}_d\|_{L^2(0,T; \BL^2(\Omega))}^2 + \dfrac{\sigma}{2}\|\Phi\|_{\BU}^2\\
		\mathrm{s.t.} \ \ \ \  &\Phi \in \BU^\mathrm{ad}, \B{v} \in \BV_{T,\BK}^{\mathrm{ad}}, \B{v} \ \mathrm{is \ a \ classical \ solution \ as \ in \ Definition \ \ref{cl_sol}} \\
	\end{aligned}
	\tag{${P}_{T}^\BK$}\label{minp}
\end{equation}
where $T>0$, the compact set $\BK$ is \emph{suitable} (see Definition \ref{suitable}), $\B{v}_d \in \mathrm{int}(\BV_{T,\BK}^\mathrm{ad})$ is a given desired (gas) distribution, and $\sigma>0.$ 

Our main theorem is the first step needed to study \eqref{minp} in the sense that it establishes its consistency, i.e., that there exists a time $T>0$ such that the set of pairs $(\Phi,\B{v})$ with $\B{v}$ a classical solution is non-empty. 
\begin{theorem}[\bf Well-posedness of the state system]\label{mainwpp} 	
	Let $\BCK$ be a suitable compact set and $\B{v}_\mathrm{e}$ be a steady state solution of the semilinear Euler system \eqref{ksys} such that Assumption \ref{SSnet} is satisfied and let $\Phi \in \BU^\mathrm{ad}$. Then, there exist $T = T(\kappa_\BU, \BCK, \B{v}_{\mathrm{e}}) > 0$ and a unique $\B{v} = \B{v}(\Phi)  \in \BBX_T^{\BCK}(\BH^1(\Omega),\BL^2(\Omega))$ such that $\B{u} = \B{v} - \BB_1\Phi$ is a classical solution of \eqref{abs1}.
\end{theorem} 
Inspired by \cite{H-H-S-S}, the solution in Theorem \ref{mainwpp} is built on semigroup theory (see \cite{E-R,P}). However, in \cite{H-H-S-S} the nonlinearity is not explicit and regularity properties such as Lipschitz continuity (with respect to the solution) are assumed, allowing the authors to apply standard results of \cite{P} to obtain well-posedness. In our case, the nonlinearity is explicit and can possibly degenerate. The scope of our proof is not similar to the treatment of semilinear problems in \cite{P}, i.e., solutions are constructed as fixed points of contraction-maps whose formula come from the representation of the underlying linear abstract initial value Cauchy problem. The main reason is invariance: even on spaces away from the vacuum where some sort of Lipschitz continuity (w.r.t to the solution) is guaranteed for our nonlinearity, it is hard to construct sets on which a contraction \emph{self}-map can be obtained. 

Also, it is known that the interplay of smallness of time and data is standard. However, in our case it seems that smallness alone is not enough. As a consequence, our strategy lies in taking advantage of the structure of the initial data (taken as steady state solution) along with geometric properties provided by \emph{suitable} (in the sense of Definition \ref{suitable}) compact sets.

Theorem \ref{mainwpp} allows us to define a control-to-state map \begin{equation}\label{ctsmap}\BBS: \BU^\mathrm{ad} \subset \BU \to \BBX_T^\BCK\left(\BH^1(\Omega),\BL^2(\Omega)\right).\end{equation} Due to the embedding $\BBX_T\left(\BH^1(\Omega),\BL^2(\Omega)\right) \hookrightarrow \BV_T$, $\BBS$ may also be taken as a map from $\BU^\mathrm{ad}$ to both $L^2(0,T; \BL^2(\Omega))$ and $\BV_T$, respectively. However, for our later goal of studying differentiability properties of the control-to-state map, we need to restrict this map to $\BU_0^\eta$, i.e., we consider $$\BS: \BU_0^\eta \subset \BU_0 \to \BBX_T^\BK\left(\BH^1(\Omega),\BL^2(\Omega)\right)$$ with $\BS(\Phi) := \BBS(\phi(\Phi)) = \BBS(\Phi + \Phi^\mathrm{e})$; see Remark \ref{U0eta}. 

Consequently, the reduced version of \eqref{minp} can be written as 
\begin{equation}
	\begin{aligned}
		\min\limits \ &  J(\Phi) \coloneqq \hat J\left(\Phi,\BS(\Phi)\right) =  \dfrac{1}{2}\|\BS(\Phi)-\B{v}_d\|_{L^2(0,T;\BL^2(\Omega))}^2 + \dfrac{\sigma}{2}\|\Phi\|_{\BU}^2\\
		\mathrm{s.t.} \ \  &\Phi \in \BU_0^\eta \ \mathrm{and} \ \BS(\Phi) \in \BV_{T,\BK}^\mathrm{ad}.\\
	\end{aligned}
	\tag{${rP}_{T}^\BK$}\label{minp4}
\end{equation}
The next goal is to obtain existence of optimal control for \eqref{minp4}. This is often done by the direct method of the calculus of variations. For its application we need to establish several properties of the control-to-state map. Section \ref{existence_control} is dedicated to the corresponding analysis. The main property which renders the direct method applicable, is the following one:
\begin{proposition}
	The map $\BS: \BU_0^\eta \to L^2(0,T; \BL^2(\Omega))$ is weak-to-strong continuous.
\end{proposition}
The existence of optimal controls then follows.
\begin{theorem}[\bf Existence of optimal controls]\label{existoptimal}
	The problem \eqref{minp4} admits an optimal solution.
\end{theorem}

Next we seek to establish a first-order characterization of optimal solutions. Such conditions are useful as they typically form the basis for developing numerical solution algorithms. For the pertinent first-order analysis we need to study differentiability properties of $\BS$ as well as the existence and regularity of adjoint states. Both notions are essential for establishing so-called primal-dual first-order optimality conditions relying on (bounded) Lagrange multipliers. 

\begin{proposition}[\bf Differentiability of $\BS$]
	 The map $\BS$ is continuously G\^{a}teaux differentiable when defined from $\BU_0^\eta$ to either $L^2(0,T;\BL^2(\Omega))$ or $\BV_T$. 
\end{proposition}
Next, we want to characterize the derivative $\BS'(\Phi)$. For $\Phi \in \BU^{\mathrm{ad}}$ and $\BCK$ a suitable compact set, let $\BBS(\Phi) = ((p^k(\Phi),q^k(\Phi))\kk{k} \in \BBX_{T}^\BCK(\BH^1(\Omega),\BL^2(\Omega))$ and define 
$$\calF^k(\Phi) \coloneqq 
\begin{bmatrix}
	0 & 0 \\ 
	\beta_k \dfrac{q^k(\Phi)|q^k(\Phi)|}{p^k(\Phi)^2} & -2\beta_k \dfrac{|q^k(\Phi)|}{p^k(\Phi)}
\end{bmatrix}$$ 
as well as $\BFF(\Phi) \coloneqq (\calF^k(\Phi))\kk{k}.$ Now, for a given $\Phi \in \BU_0^\eta$ and $\B{h} \in \BU_0$, we let $\tilde \Phi \coloneqq \Phi + \Phi^\mathrm{e}$ and define the affine map $\B{f}_\Phi(\B{h}): [0,T] \times \BL^2(\Omega) \to \BL^2(\Omega)$ as $\B{f}_\Phi(\B{h})(t,\B{w}) \coloneqq \BFF(\tilde \Phi)\B{w} + (\BB + \BFF(\tilde \Phi)\BB_1)\B{h}$. 
\begin{theorem}\label{diffe_main}
	For each $\Phi \in \BU_0^\eta$ we have that $\BS'(\Phi,\cdot)$ is a linear and continuous operator from $\BU_0$ to $L^2(0,T;\BL^2(\Omega))\cap \BV_T$, i.e.,  $\BS'(\Phi,\cdot)\in \calL(\BU_0,L^2(0,T;\BL^2(\Omega))\cap \BV_T)$ given by
	\begin{equation}
		\BBS'(\Phi,\B{h}) = \B{w}_\Phi(\B{h}) + \BB_1\B{h}, \qquad \B{h} \in \BU_0,
	\end{equation}
	where $\B{w}_\Phi(\B{h})$ is the solution of the first-order system
	\begin{subnumcases}{\label{dern1}}
		\B{w}_t = (\BA+\BP)\B{w}+\B{f}_\Phi(\B{h})(t,\B{w}),\label{derna} \\[2mm]
		\B{w}(0) = 0.\label{dernb} 
	\end{subnumcases} 
	Moreover, the map $\Phi \to \BS'(\Phi,\cdot)$ is (Lipschitz) continuous from $\BU_0^\eta$ to $\calL(\BU_0,L^2(0,T;\BL^2(\Omega))).$
\end{theorem}
Next, we establish the existence of adjoint states which allows us to characterize the gradient of the cost function $J$. First, for $\Phi \in \BU^{\mathrm{ad}}$ and $\BCK$ a suitable compact set define $\BFF_\square(\Phi) \coloneqq \BFF(\Phi)^\ast$ where the latter is taken as a (bounded) operator from $\BL^2(\Omega)$ to itself. Further define the affine map $\B{g}_\Phi: [0,T] \times \BL^2(\Omega) \to \BL^2(\Omega)$ as $\B{g}_\Phi(t,\B{p}) \coloneqq \BFF_\square(\tilde\Phi)\B{p} + \BS(\Phi) - \B{v}_d, (\tilde \Phi = \Phi + \Phi^\mathrm{e}).$
\begin{proposition}\label{adjoint} There exists a unique solution $\B{p} \coloneqq \B{p}_\Phi \in C([0,T];\BL^2(\Omega))$, called the {\bf adjoint state}, to the problem
	\begin{subnumcases}{\label{dder6}}
		-\B{p}_t = (-\BA+\BP^\ast)\B{p}+\B{g}_\Phi(t,\B{p}), \label{dder7} \\[2mm]
		\B{p}(T) = 0.\label{dder8} 
	\end{subnumcases} 
\end{proposition}

Utilizing $\B{p}_\Phi$, we obtain the following representation of the gradient $J'(\Phi)$.
\begin{corollary}\label{gradient_mrd}
	Let $\Phi \in \BU_0^\eta$ and $\BCK$ be a suitable compact set. Let $\B{p}_\Phi \in C([0,T];\BL^2(\Omega))$ be the adjoint state of Proposition \ref{adjoint}. Then the following formula holds:
	\begin{equation}\label{jgrad_mrd}
		J'(\Phi) = (\BB^\ast + \BB_1^\ast\BFF_\square(\tilde\Phi))\B{p}_\Phi - \BB_1^\ast(\BS(\Phi) - \B{v}_d) + \sigma \Phi,
	\end{equation}
	where $\BB^\ast$ and $\BB_1^\ast$ are the adjoints of $\BB and \BB_1$, when both are considered as bounded operators from $\BU$ to $L^2(0,T;\BL^2(\Omega))$. 
\end{corollary}

Our final goal is the characterization of an optimal control according to Theorem \ref{existoptimal} via first-order optimality conditions. Notice that if there are no state constraints, i.e. $\BCK = \BK$ in Theorem \ref{mainwpp}, then such conditions follow directly from Fermat's theorem and \eqref{jgrad_mrd}. %We do not pursue this any further in this paper.
%Instead, we assume non-redundant state contraints and use the notion of {\bf normality} introduced e.g. in \cite{B-C2,B-C,C2}. 
However, the case of non-redundant state constraints is significantly more delicate. For deriving first-order conditions in this case we rely on the notion of {\bf normality} (see, e.g., \cite{B-C2,B-C,C2}) which we introduce next in our context. 

Notice first that problem \eqref{minp4} is of the form 
\cite[Problem Q, p. 1001]{C}.  Then, for any given locally optimal control $\Phi^\ast \in \BU_0^\eta$, it follows by \cite[Theorem 5.2, p. 1001]{C} that 
\begin{itemize}
	\item[\bf (i)] $\Phi^\ast$ is feasible, i.e., $\Phi^\ast \in \BU_0^\eta$ and $\BS(\Phi) \in \BV_+^2$;
	\item[\bf (ii)] there exists an unique adjoint state $\B{p}_{\Phi^\ast}$ associated to $\Phi^\ast;$
	\item[\bf (iii)] there exists a measure $\Lambda \in \calM(Q_T)^2$ such that the inequality 
	\begin{equation}\label{measure_ine}
		\int_{Q_T^2} (\B{y} - \BS(\Phi))d\Lambda \leqslant 0
	\end{equation}
	holds for all $\B{y} \in \BV_+^2$;
	\item[\bf (iv)] there exists a real number $\zeta \geqslant 0$ such that the inequality 
	\begin{align}\label{main_var_ine_mrd}
		\zeta\left((\BB^\ast + \BB_1^\ast\BFF_\square(\Phi^\ast))\B{p}_{\Phi^\ast} - \BB_1^\ast(\BS(\Phi^\ast) - \B{v}_d) + \sigma \Phi, \Phi-\Phi^\ast\right) \nonumber \\ + \int_{Q_{T}^2} \BS'(\Phi^\ast,(\Phi-\Phi^\ast))d\Lambda \geqslant 0
	\end{align}
	holds for all $\Phi \in \BU_0^\eta.$
\end{itemize}

\begin{definition}
	We say that problem \eqref{minp4} is {\bf normal} at $\Phi^\ast$ provided we can take $\zeta = 1$ in \emph{\eqref{main_var_ine_mrd}}.
\end{definition}

Most of the sufficient conditions to ensure normality involve some sort of regularity condition of the feasible set, a so-called {\it constraint qualification}. The {\bf Robinson-Zowe-Kurczyusz (RZK) regularity condition} is particularly useful for our purposes. It requires the existence of a direction $\B{h} \in \BU_0$ such that 
\begin{equation}\label{rob_CQ}
	\BS(\Phi^\ast) + \BS'(\Phi^\ast,\B{h}) \in \mathrm{int}(\BV_+^2).
\end{equation}
It is slightly weaker than Slater's condition which requires $\BS(\Phi^\ast) \in \rm{int}(\BV_+^2)$ or, for instance, the surjectivity of $\BS'(\Phi^\ast,\cdot)$, both of which imply \eqref{rob_CQ}, but none of these can be guaranteed in our setting. Hence, in general, one cannot guarantee that our problem is normal. We can, however, prove that normality holds almost surely in the sense explained next. 

Since $\mathrm{int}(\BV_{T,\BK}^\mathrm{ad}) \neq \emptyset$, let $\B{v}_0 \in \mathrm{int}(\BV_{T,\BK}^\mathrm{ad})$ and for every $\delta > 0$ define the perturbed convex set 
\begin{equation}\label{Cdelta}
	\BV_{\BK}^\delta \coloneqq (1-\delta)\B{v}_0 + \delta \BV_{T,\BK}^\mathrm{ad}.
\end{equation} 
Now, consider the perturbed problem 
\begin{equation}
	\begin{aligned}
		\min\limits \ &  J(\Phi) \ \mathrm{s.t} \ \Phi \in \BU_0^\eta \ \mathrm{and} \ \BS(\Phi) \in \BV_{\BK}^\delta.\\
	\end{aligned}
	\tag{${rP}_{T}^{\BK,\delta}$}\label{minpp4}
\end{equation}
We then have the following general result, which is an adaptation of \cite[Theorem 2.2, p. 73]{B-C2}.

\begin{theorem}\label{asdelta}
	Let $I \subseteq \BR_+^\ast$ be an interval such that for all $\delta \in I$ the perturbed problem \eqref{minpp4} admits a solution $\Phi_\delta$. Then problem \eqref{minpp4} is normal at $\Phi_\delta$ for almost every $\delta \in I$.
\end{theorem}

As a corollary, we have the final result of this paper.
\begin{theorem}\label{aseps}
	The problem \eqref{minp4} is {\bf almost surely} normal at $\Phi^\ast$ in the following sense: given $\varepsilon>0$, there exists $\delta >0$ such that $|1-\delta| < \varepsilon$ and problem \eqref{minpp4} is normal at $\Phi_\delta.$
\end{theorem}
\begin{proof}
	Apply Theorem \ref{asdelta} with problem \eqref{minp4} perturbed around $\B{v}_\mathrm{e} \in \mathrm{int}(\BV_{T,\BK}^\mathrm{ad}).$
\end{proof}
\begin{remark}
	Theorem \ref{aseps} would still be true if we did not have an interior feasible point to perturb $\BV_{T,\BK}^\mathrm{ad}$ with. Indeed we see, by inspection of the proof of existence in Theorem \ref{existoptimal}, that problem \eqref{minpp4} will have a solution as long as $\BV_{\BK}^\delta$ contains a feasible point. 
	Let $\B{v}_0$ be any interior point of $\BV_{T,\BK}^\mathrm{ad}$ (not necessarity feasible) and $\overline{\B{v}}$ be any feasible point (not necessarily interior). 
	
	Of course if $\delta = 1$ we have a feasible point: $\overline{\B{v}}.$ If $\delta \neq 1$, one way to guarantee nonemptiness of the feasible set is to find $\B{v} \in \BV_{T,\BK}^{ad}$ such that $(1-\delta)\B{v}_0 + \delta \B{v} = \overline{\B{v}}$ for which it suffices to show that $\delta^{-1}(\delta-1)\B{v}_0 + \delta^{-1}\overline{\B{v}} \in \BV_{T,\BK}^{ad}.$
	For $\delta > 1$ the above is true by convexity. One cannot go below $1$ unless we assume $\overline{\B{v}}$ is interior, but we do not need the latter as Theorem \ref{asdelta} already implies that normality holds for almost every $\delta \in [1,+\infty)$. 
\end{remark}

\section{Local well-posedness of the state system}\label{localwell}

In this section we show that the state system, i.e. the semilinear Euler system with the Kirchhoff law, admits a classical solution. Our proof strategy relies on the following two main steps: \begin{itemize}
	\item First we study the well-posedness of the linear system with fixed (or frozen) non-linearity, as well as its regularity and continuity properties;
	
	\item then, via an iterative process, we construct a sequence of solutions (to associated linear problems) that will converge to the solution of the general semilinear system, invoking certain structural conditions.
\end{itemize}

Due to the hyperbolic character of the system, shock discontinuities are likely to occur in finite time. To avoid this (since we work with smooth initial and boundary data) we use a geometric argument to construct a small time $T_\ast$ during which solutions will remain smooth. We briefly explain the general idea.

Let us introduce \textit{another} suitable compact set $\tilde{\mathscr{K}}$ such that $\tilde{\mathscr{K}} \subset \mathscr{K}$ and $\B{v}_\mathrm{e} \in \tilde{\mathscr{K}}$. Take $r = r(\tilde{\mathscr{K}}, \mathscr{K}) > 0$ as the largest real number such that 
\begin{equation}\label{rKK}
	\|\B{v}-\B{v}_e\|_{C(\overline\Omega)} < r \Longrightarrow \B{v} \in \mathscr{K}.
\end{equation}
Notice that due to the definition of suitable compact sets, we can always find such a bound $r$. If $\B{v} \in \BX_T(\BH^1(\Omega),\BL^2(\Omega))$ is such that $\B{v}(0,\cdot) = \B{v}_{\mathrm{e}}$, then there exists a (maximal) time $\tilde T>0$ such that $\B{v}(t,\cdot) \in B_r(\B{v}_\mathrm{e})$ for all $0 \leqslant t \leqslant \tilde T.$

The main goal, therefore, is to take a sequence $\{\B{v}_k\}$ of such functions with their corresponding $\{\tilde T_k\}$ converging to a smooth solution $\B{v}$ of the semilinear problem. The most difficult task is to show that the sequence $\{\tilde{T_k}\}$ admits a uniform lower bound that is positive.
\begin{remark}\label{notationestimate}
	Most of the proofs in this section will be based on estimates. To avoid overloaded notation, we use the symbol $\lesssim$ to indicate the presence of a constant that might depend on the fixed constants of the problem, but that do not play an important role in that particular step. 
\end{remark}
\subsection{Linear analysis} \label{seclin}
We start with the main property of the differential operator $\BA$ defined in \eqref{defAP}.
\begin{theorem}\label{maxmon} 
	The operator $\BA: \calD(\BA) \subset {\BL^2(\Omega)} \to {\BL^2(\Omega)}$ is maximal dissipative.
\end{theorem}
\begin{proof}
	Let $\B{u} \in \calD(\BA)$ and let $P_v$ be the representative of the pressure at node $v \in \calV^\circ.$ By the continuity conditions and considering the scaling of the $\BL^2$-norm we compute
	\vspace{-.35cm}
	\begin{align}
		(\BA \B{u}, \B{u})_{{\BL^2(\Omega)}} & = \sum\limits_{k = 1}^{m} (\calA_k \vec{u}_k,\vec{u}_k)_{{\BBL^2(\Omega_k)}} \nonumber \\ 
		&= \sum\limits_{k = 1}^{m} D_k^2\left[(-c^2\partial_x \tilde q^k,\tilde p^k)_{L^2(\Omega_k)} + c^2(-\partial_x \tilde p^k,\tilde q^k)_{L^2(\Omega_k)}\right] \nonumber\\ 
		&= -c^2\sum\limits_{k = 1}^{m} D_k^2\left[\int_0^{L_k}\left(\partial_x \tilde q^k(x)\tilde p^k(x) + \partial_x \tilde p^k(x)\tilde q^k(x)\right)dx\right] \nonumber \\ 
		&= -c^2\sum\limits_{k = 1}^{m} D_k^2\tilde p^k(x)\tilde q^k(x) \biggr\rvert_{0}^{L_k} \label{c1} \nonumber \\ 
		&= -c^2\left\{\sum\limits_{k = 1}^{m} D_k^2\tilde p^k(L_k)\tilde q^k(L_k)-\sum\limits_{k = 1}^{m} D_k^2\tilde p^k(0)\tilde q^k(0)\right\} \nonumber \\ 
		& = -c^2\left\{\sum\limits_{i=1}^n\left[\sum\limits_{\substack{k=1 \\ \xi_k(v_i) = 1 \\ v_i \notin \calV_+^\partial}}^{m} D_k^2\tilde p^k(L_k)\tilde q^k(L_k)-\sum\limits_{\substack{k=1 \\ \xi_k(v_i) = -1 \\ v_i \notin \calV_-^\partial}}^{m} D_k^2\tilde p^k(0)\tilde q^k(0)\right]\right\} \nonumber \\ 
		& = -c^2\left\{\sum\limits_{i=1}^nP_{v_i}\left[\sum\limits_{\substack{k=1 \\ \xi_k(v_i) = 1 \\ v_i \notin \calV_+^\partial}}^{m} D_k^2\tilde q^k(L_k)-\sum\limits_{\substack{k=1 \\ \xi_k(v_i) = -1\\\\ v_i \notin \calV_-^\partial}}^{m} D_k^2\tilde q^k(0)\right]\right\} \nonumber \\ 
		& = -c^2\left\{\sum\limits_{\substack{i=1 \\ v_i \notin \calV^\partial}}^nP_{v_i}\left[\sum\limits_{k=1}^{m} \xi_k(v_i) D_k^2\tilde q^k(v_i)\right]\right\} \nonumber \\ 
		& = -c^2 \sum\limits_{v \in \calV^\circ} P_v \sum\limits_{k = 1}^{m} \xi_k(v) D_k^2 \tilde q^k(v) = 0, 
	\end{align} 
	whereby $\BA$ is dissipative. By using the (equivalent) sum norm, it follows that $\|\BA\B{u}\|_{{\BL^2(\Omega)}} =c\left\{\|\B{u}\|_{{\BH^1(\Omega)}} - \|\B{u}\|_{{\BL^2(\Omega)}}\right\}$, and then maximimality follows from, e.g, \cite[Proposition 6.55, p. 314]{E-G}.
\end{proof}
The exact same proof as above implies that $-\BA: \calD(\BA) \subset \BL^2(\Omega) \to \BL^2(\Omega)$ is maximal dissipative. This in turn yields the following corollary which is of particular interest for control problems.
\begin{corollary}\label{skew}
	The operator $\BA$ is skew--adjoint, i.e., $\BA^\ast = -\BA.$
\end{corollary}
\begin{proof}
	First recall that a closed operator (in Hilbert space) is maximal dissipative if and only if its dual is dissipative. We are going to use this fact.
	
	We know that $\BA$ and $-\BA$ are maximal dissipative. Given $\B{u}, \B{v} \in \calD(\BA)$ we have, in light of Theorem \ref{maxmon}, that 
$
	(\BA\B{u},\B{v}) + (\BA\B{v},\B{u}) = (\BA(\B{u}+\B{v}),\B{u}+\B{v}) - (\BA\B{u},\B{u}) - (\BA\B{v},\B{v}) = 0.
$
Therefore, $\BA^\ast$ is an extension of $-\BA.$ We now show that $-\BA$ is also an extension of $\BA^\ast.$ 

Let $\B{u} \in \calD(\BA^\ast)$ with $\B{f} = \BA^\ast \B{u}$ and let $\B{g} = \B{u}-\B{f} \in \BL^2(\Omega).$ Since $-\BA$ is maximal dissipative, there exists $\B{v} \in \calD(\BA)$ such that $\B{v} - (-\BA)\B{v} = \B{g}$. But since $\BA^\ast$ is an extension of $-\BA$ it follows that $\B{v} \in \calD(\BA^\ast)$ and $\B{v} - \BA^\ast \B{v} = \B{g} = \B{u} - \BA^\ast\B{u}$, whereby $\B{v}-\B{u} - \BA^\ast(\B{v}-\B{u}) = 0$. Now the maximal dissipativity of $-\BA$ gives the dissipativity of $-\BA^\ast$. Then taking the $\BL^2$--inner product of the above identity with $\B{v}-\B{u}$ gives $\B{u}=\B{v} \in \calD(\BA).$ Then $-\BA$ extends $\BA^\ast$ and the proof is complete.
\end{proof}

The Lumer-Phillips Theorem \cite[Theorem 4.3, p. 14]{P} now implies that $\BA$ generates in ${\BL^2(\Omega)}$ a $C_0$--group of isometries $\{\tilde \BT(t)\}_{t \in \BR}.$ Moreover, since $\BP \in \calL({\BL^2(\Omega)})$ it follows from \cite[Theorem 1.1, p. 76]{P} that $\BA+\BP$ generates a $C_0$--semigroup $\{\BT(t)\}_{t \geqslant 0}$ and from \cite[Corollary 1.3, p. 78]{P} that 
\begin{equation}\label{semig}
	\|\tilde \BT(t) - \BT(t)\| \leqslant e^{\|P\| t} -1.
\end{equation}

\begin{lemma}\label{regB}
	For $s = 0$ or $s=1$ we have $\BB \in \mathcal{L}\left(\BU;H^1(0,T;\BH^s(\Omega))\right)$.
\end{lemma}
\begin{proof}
	The proof is rather elementary. However, we present it here so we can fix quantities that will be later used when estimating the nonlinearity. For a given $\Phi \in \BU$ we denote $\Phi = (\Phi^k)$ with each $\Phi^k = (\Phi_1^k,\Phi_2^k).$ Let $s \in \mathbb{N} \cup \{0\}$.
	\begin{itemize}
		\item For $\BB_0$ we have, for each $k$
		\begin{align*}
			\|\BBB_0^k\Phi^k\|_{L^2(0,T;\BBH^s(\Omega_k))}^2 
			& = \sum\limits_{m \leqslant s} \|D_x^m\BBB_0^k\Phi^k\|_{L^2(0,T;\BBL^2(\Omega_k))}^2 \\
			&= \sum\limits_{m \leqslant s} \int_0^T \|D_x^m\BBB_0^k\Phi^k(t,\cdot)\|_{\BBL^2(\Omega_k)}^2dt \\ 
			&= \sum\limits_{m \leqslant s} \dfrac{c^2}{L_k^2}\int_0^T D_k^2\left[c^2\int_0^{L_k} |\partial_x^s\Phi_1^k(t)|^2dx + \int_0^{L_k}|\partial_x^s\Phi_2^k(t)|^2dx\right]dt \\
			&= \dfrac{c^2}{L_k^2}\int_0^T D_k^2\left[c^2\int_0^{L_k} |\Phi_1^k(t)|^2dx + \int_0^{L_k}|\Phi_2^k(t)|^2dx\right]dt \\
			&= \dfrac{c^2D_k^2}{L_k}\int_0^T \left[c^2|\Phi_1^k(t)|^2+ |\Phi_2^k(t)|^2\right]dt
			 \lesssim \|\Phi^k\|_{L^2(0,T) \times L^2(0,T)}^2
		\end{align*} 
		and
		\begin{align*}
			\|D_t\BBB_0^k\Phi^k\|_{L^2(0,T;\BBH^s(\Omega_k))}^2 
			& = \sum\limits_{m \leqslant s} \|D_x^mD_t\BBB_0^k\Phi^k\|_{L^2(0,T;\BBL^2(\Omega_k))}^2 \\
			&= \sum\limits_{m \leqslant s} \int_0^T \|D_x^mD_t\BBB_0^k\Phi^k(t,\cdot)\|_{\BBL^2(\Omega_k)}^2dt \\ 
			&= \sum\limits_{m \leqslant s} \dfrac{c^2}{L_k^2}\int_0^T D_k^2\left[c^2\int_0^{L_k} \left|\partial_x^s\dfrac{d}{dt}\Phi_1^k(t)\right|^2dx + \int_0^{L_k}\left|\partial_x^s\dfrac{d}{dt}\Phi_2^k(t)\right|^2dx\right]dt \\
			&= \dfrac{c^2}{L_k^2}\int_0^T D_k^2\left[c^2\int_0^{L_k} \left|\dfrac{d}{dt}\Phi_1^k(t)\right|^2dx + \int_0^{L_k}\left|\dfrac{d}{dt}\Phi_2^k(t)\right|^2dx\right]dt \\
			&= \dfrac{c^2D_k^2}{L_k}\int_0^T \left[c^2\left|\dfrac{d}{dt}\Phi_1^k(t)\right|^2+ \left|\dfrac{d}{dt}\Phi_2^k(t)\right|^2\right]dt
			 \lesssim \|D_t\Phi^k\|_{L^2(0,T) \times L^2(0,T)}^2
		\end{align*}
		whereby 
		\begin{align*}
			\|\BBB_0^k\Phi^k\|_{H^1(0,T;\BBH^s(\Omega_k))}^2 
			& \lesssim \|\Phi^k\|_{H^2(0,T) \times H^2(0,T)}^2.
		\end{align*} 
		By taking the maximum of the $k$ constants above, we obtain $\|\BB_0\|_{\mathcal{L}(\BU;H^1(0,T;\BH^s(\Omega)))}.$
		
	\item  For $\BB_1$ we have, for each $k$,
	\begin{align*}
		&\|\BBB_1^k\Phi^k\|_{L^2(0,T;\BBH^s(\Omega_k))}^2 
		= \sum\limits_{m \leqslant s} \|D_x^m \BBB_1^k \Phi^k\|_{L^2(0,T;\BBL^2(\Omega_k))}^2
		= \sum\limits_{m \leqslant s} \int_0^T \|D_x^m \BBB_1^k\Phi^k(t,\cdot)\|_{\BBL^2(\Omega_k)}^2dt \\ 
		&=\sum\limits_{m \leqslant s}  \dfrac{1}{L_k^2}\int_0^T D_k^2\left[\int_0^{L_k} |\partial_x^m[(L_k-x)\Phi_1^k(t)]|^2dx + c^2\int_0^{L_k}|\partial_x^m[x\Phi_2^k(t)]|^2dx\right]dt \\
		&=\dfrac{1}{L_k^2}\int_0^T D_k^2\left[\int_0^{L_k} |(L_k-x)\Phi_1^k(t)|^2dx + c^2\int_0^{L_k}|x\Phi_2^k(t)|^2dx\right]dt \\
		&+  \dfrac{1}{L_k^2}\int_0^T D_k^2\left[\int_0^{L_k} |\Phi_1^k(t)|^2dx + c^2\int_0^{L_k}|\Phi_2^k(t)|^2dx\right]dt \\
		& \leqslant \dfrac{1}{L_k^2}\int_0^T D_k^2\left[\left(\max\limits_{x \in \overline{\Omega_k}}|L_k-x|\right)^2\int_0^{L_k} |
		\Phi_1^k(t)|^2dx +c^2\left(\max\limits_{x \in \overline{\Omega_k}}|x|\right)^2 \int_0^{L_k}|\Phi_2^k(t)|^2dx\right]dt \\
		&+  \dfrac{1}{L_k^2}\int_0^T D_k^2\left[\int_0^{L_k} |\Phi_1^k(t)|^2dx + c^2\int_0^{L_k}|\Phi_2^k(t)|^2dx\right]dt
		\lesssim \|\Phi^k\|_{L^2(0,T) \times L^2(0,T)}^2
	\end{align*} 
	and
	\begin{align*}
		&\|D_t\BBB_1^k\Phi^k\|_{L^2(0,T;\BBH^s(\Omega_k))}^2 
		= \sum\limits_{m \leqslant s} \|D_x^m D_t \BBB_1^k \Phi^k\|_{L^2(0,T;\BBL^2(\Omega_k))}^2
		= \sum\limits_{m \leqslant s} \int_0^T \|D_x^m D_t\BBB_1^k\Phi^k(t,\cdot)\|_{\BBL^2(\Omega_k)}^2dt \\ 
		&=\sum\limits_{m \leqslant s}  \dfrac{1}{L_k^2}\int_0^T D_k^2\left[\int_0^{L_k} \left|\partial_x^m(L_k-x)\dfrac{d}{dt}\Phi_1^k(t)\right|^2dx + c^2\int_0^{L_k}\left|\partial_x^m(x)\dfrac{d}{dt}\Phi_2^k(t)\right|^2dx\right]dt \\
		&= \dfrac{1}{L_k^2}\int_0^T D_k^2\left[\int_0^{L_k} \left|(L_k-x)\dfrac{d}{dt}\Phi_1^k(t)\right|^2dx + c^2\int_0^{L_k}\left|x\dfrac{d}{dt}\Phi_2^k(t)\right|^2dx\right]dt \\
		&+ \dfrac{1}{L_k^2}\int_0^T D_k^2\left[\int_0^{L_k} \left|\dfrac{d}{dt}\Phi_1^k(t)\right|^2dx + c^2\int_0^{L_k}\left|\dfrac{d}{dt}\Phi_2^k(t)\right|^2dx\right]dt
		 \lesssim \|\Phi^k\|_{H^1(0,T) \times H^1(0,T)}^2
	\end{align*} 
	whereby 
\begin{align*}
	\|\BBB_0^k\Phi^k\|_{H^1(0,T;\BBH^s(\Omega_k))}^2 
	& \lesssim \|\Phi^k\|_{H^2(0,T) \times H^2(0,T)}^2
\end{align*} 
and again, by taking the maximum of $k$ constants, we obtain $	\|\BB_1\|_{\mathcal{L}(\BU;H^1(0,T;\BH^s(\Omega)))}$.
\item For $D_t\BB_1$ we have, for each $k$, as before, 
\begin{align*}
	\|D_t\BBB_1^k\Phi^k\|_{L^2(0,T;\BBH^s(\Omega_k))}^2 
	\lesssim \|\Phi^k\|_{H^1(0,T) \times H^1(0,T)}^2
\end{align*} 
and
\begin{align*}
	&\|D_t^2\BBB_1^k\Phi^k\|_{L^2(0,T;\BBH^s(\Omega_k))}^2 
	= \sum\limits_{m \leqslant s} \|D_x^m D_t^2 \BBB_1^k \Phi^k\|_{L^2(0,T;\BBL^2(\Omega_k))}^2
	= \sum\limits_{m \leqslant s} \int_0^T \|D_x^m D_t^2\BBB_1^k\Phi^k(t,\cdot)\|_{\BBL^2(\Omega_k)}^2dt \\ 
	&=\sum\limits_{m \leqslant s}  \dfrac{1}{L_k^2}\int_0^T D_k^2\left[\int_0^{L_k} \left|\partial_x^m(L_k-x)\dfrac{d^2}{dt^2}\Phi_1^k(t)\right|^2dx + c^2\int_0^{L_k}\left|\partial_x^m(x)\dfrac{d^2}{dt^2}\Phi_2^k(t)\right|^2dx\right]dt \\
	&= \dfrac{1}{L_k^2}\int_0^T D_k^2\left[\int_0^{L_k} \left|(L_k-x)\dfrac{d^2}{dt^2}\Phi_1^k(t)\right|^2dx + c^2\int_0^{L_k}\left|x\dfrac{d^2}{dt^2}\Phi_2^k(t)\right|^2dx\right]dt \\
	&+ \dfrac{1}{L_k^2}\int_0^T D_k^2\left[\int_0^{L_k} \left|\dfrac{d^2}{dt^2}\Phi_1^k(t)\right|^2dx + c^2\int_0^{L_k}\left|\dfrac{d^2}{dt^2}\Phi_2^k(t)\right|^2dx\right]dt
	\lesssim \|\Phi^k\|_{H^2(0,T) \times H^2(0,T)}^2
\end{align*} 
whereby 
\begin{align*}
	\|D_t\BBB_1^k\Phi^k\|_{H^1(0,T;\BBH^s(\Omega_k))}^2 
	& \lesssim \|\Phi^k\|_{H^2(0,T) \times H^2(0,T)}^2
\end{align*} 
from where we obtain $\|D_t\BB_1\|_{\mathcal{L}(\BU;H^1(0,T;\BH^s(\Omega)))}$.
		\item For $\BP$ we have, for each $k$
	\begin{align*}
		\|\BBP_k\Phi^k\|_{L^2(0,T;\BBH^s(\Omega_k))}^2 
		& = \sum\limits_{m \leqslant s} \|D_x^m\BBP_k\Phi^k\|_{L^2(0,T;\BBL^2(\Omega_k))}^2 \\
		&= \sum\limits_{m \leqslant s} \int_0^T \|D_x^m\BBP_k\Phi^k(t,\cdot)\|_{\BBL^2(\Omega_k)}^2dt \\ 
		&= \sum\limits_{m \leqslant s} \int_0^T D_k^2c^2\int_0^{L_k} |\partial_x^s\gamma_k\Phi_1^k(t)|^2dxdt \\
		&= D_k^2\gamma_k^2L_kc^2\int_0^T |\Phi_1^k(t)|^2dt 
		 \lesssim \|\Phi^k\|_{L^2(0,T) \times L^2(0,T)}^2
	\end{align*} 
	and similarly,
		\begin{align*}
		\|D_t\BBP_k\Phi^k\|_{L^2(0,T;\BBH^s(\Omega_k))}^2 
		& \lesssim \|\Phi^k\|_{H^1(0,T) \times H^1(0,T)}^2
	\end{align*} 
	whereby 
	\begin{align*}
		\|\BBP_k\Phi^k\|_{H^1(0,T;\BBH^s(\Omega_k))}^2 
		& \lesssim \|\Phi^k\|_{H^2(0,T) \times H^2(0,T)}^2
	\end{align*} 
	from where we get $\|\BP\|_{\mathcal{L}(\BU;H^1(0,T;\BH^s(\Omega)))}.$
\item Finally, for $\BP\BB_1$ we have, for each $k$
\begin{align*}
	&\|\BBP_k\BBB_1\Phi^k\|_{L^2(0,T;\BBH^s(\Omega_k))}^2 
	 = \sum\limits_{m \leqslant s} \|D_x^m\BBP_k\BBB_1\Phi^k\|_{L^2(0,T;\BBL^2(\Omega_k))}^2 \\
	&= \sum\limits_{m \leqslant s} \int_0^T \|D_x^m\BBP_k\BBB_1\Phi^k(t,\cdot)\|_{\BBL^2(\Omega_k)}^2dt
	= \sum\limits_{m \leqslant s} \int_0^T \dfrac{D_k^2\gamma_k^2c^2}{L_k^2}\int_0^{L_k} |\partial_x^s(L_k-x)\Phi_1^k(t)|^2dxdt \\
	&= \int_0^T \dfrac{D_k^2\gamma_k^2c^2}{L_k^2}\int_0^{L_k} |(L_k-x)\Phi_1^k(t)|^2dxdt 
	+ \int_0^T \dfrac{D_k^2\gamma_k^2c^2}{L_k^2}\int_0^{L_k} |\Phi_1^k(t)|^2dxdt \\
	&\leqslant D_k^2\gamma_k^2c^2\left(1+\dfrac{1}{L_k}\right)\int_0^T |\Phi_1^k(t)|^2dxdt 
	 \lesssim \|\Phi^k\|_{L^2(0,T) \times L^2(0,T)}^2
\end{align*} 
and similarly,
\begin{align*}
	\|D_t\BBP_k\BBB_1\Phi^k\|_{L^2(0,T;\BBH^s(\Omega_k))}^2 
	& \lesssim\|\Phi^k\|_{H^1(0,T) \times H^1(0,T)}^2
\end{align*} 
whereby 
\begin{align*}
	\|\BBP_k\BBB_1\Phi^k\|_{H^1(0,T;\BBH^s(\Omega_k))}^2 
	& \lesssim \|\Phi^k\|_{H^2(0,T) \times H^2(0,T)}^2
\end{align*} 
from where we get $\|\BP\BB_1\|_{\mathcal{L}(\BU;H^1(0,T;\BH^s(\Omega)))}.$
\end{itemize} 
Therefore, 	since $\BB = \BB_0 + \BP\BB_1 + D_t \BB_1$, we have $\BB \in \mathcal{L}(\BU;H^1(0,T;\BH^s(\Omega))).$
\end{proof} 
We are now ready for the linear well-posedness result.
\begin{theorem}\label{linearu}
	Assume $\Phi \in \BU^\mathrm{ad}$ and let $G=G(t,\B{x})$ be such that $G \in H^1(0,T;{\BL^2(\Omega)}).$ Then for each $\B{u}_0 \in \calD(\BA)$ there exists a unique solution $\B{u} \in \BBX_T(\calD(\BA),\BL^2(\Omega))$ to the IVP
	\begin{subnumcases}{\label{abs2}}
		\B{u}_t = (\BA+\BP) \B{u} + \BB\Phi + G, \label{e232nanew2} \\[2mm]
		\B{u}(0) =\B{u}_0 ,\label{e232nbnew2} 
	\end{subnumcases}
which satisfies the following continuity estimates
	\begin{equation}\label{contest1} 
		\|\B{u}\|_{C([0,T];{\BL^2(\Omega)})} \leqslant e^{\gamma T}\left(\|\B{u}_0\|_{{\BL^2(\Omega)}}+\int_0^T\|(\BB\Phi +G)(\tau)\|_{\BL^2(\Omega)}d\tau\right),
	\end{equation} 
	\begin{equation}\label{contest3n} 
		\|\B{u}_t\|_{C([0,T];{\BL^2(\Omega)})} \lesssim e^{\gamma T}\left(\|\B{u}_1\|_{{\BL^2(\Omega)}} + \int_0^T\|D_t(\BB\Phi +G)(
	\tau)\|_{\BL^2(\Omega)}d\tau\right),
	\end{equation}  
	and
	\begin{equation}\label{contest2} 
		\|\B{u}\|_{C([0,T];\BH^1(\Omega))}\lesssim 	\|\B{u}_t\|_{C([0,T];{\BL^2(\Omega)})}
		 + \|\BB\Phi + G\|_{C([0,T];\BL^2(\Omega))},
	\end{equation}
	where $\B{u}_1 \coloneqq (\BA +\BP)\B{u}_0 + \BB\Phi(0) + G(0).$
\end{theorem} 
\begin{proof}
	First notice that, given Lemma \ref{regB}, the assumptions on $\Phi$ and $G$ imply that $\BB\Phi + G \in H^1(0,T;{\BL^2(\Omega)}).$ Then, it follows from \cite[Proposition 3.3, p. 133]{B-D-D-M} that the function $\B{u}$ defined as
	\begin{equation}\label{mildsolnet} 
		\B{u}(t) = \BT(t)\B{u}_0 + \int_0^t\BT(t-\tau)[\BB\Phi(\tau) + G(\tau)]d\tau,
	\end{equation} is a classical solution (in the sense of \cite{P}) of the problem \eqref{abs2}, i.e., $\B{u} \in \BBX_T(\calD(\BA),\BL^2(\Omega))$ and it solves \eqref{abs2} in $\BL^2(\Omega)$ and pointwise in time.
	
	Gr\"{o}nwall's inequality gives \eqref{contest1}. For \eqref{contest3n}, we differentiate \eqref{e232nanew2} with respect to time and let $\B{z} \coloneqq \B{u}_t.$ Then, since $\B{u}_1 = (\BA +\BP)\B{u}_0 + \BB\Phi(0) + G(0) \in  \BL^2(\Omega),$ we have that $\B{z} \in C([0,T];\BL^2(\Omega))$ and satisfies 
	\begin{subnumcases}{\label{td}}
		\B{z}_t = (\BA+\BP) \B{z} + D_t(\BB\Phi + G), \label{td1} \\[2mm]
		\B{z}(0) = \B{u}_1.\label{td2} 
	\end{subnumcases}
	Although the structure of \eqref{td} is very similar to \eqref{abs2}, the proof of the inequality \eqref{contest3n} requires a density argument. Let $(f_n)$ be a sequence in $C^1([0,T];\BL^2(\Omega))$ such that $f_n \to D_t(\BB\Phi + G)$ in $L^1(0,T;\BL^2(\Omega))$ and let $(\B{z}_n^0)$ be a sequence in $\calD(\BA)$ such that $\B{z}_n^0 \to \B{u}_1$ in $\BL^2(\Omega).$ Associated to this data, let $(\B{z}_n)$ be the sequence in $\BBX_T(\calD(\BA),\BL^2(\Omega))$ such that, for each $n$, $\B{z}_n$ solves
	\begin{subnumcases}{\label{tdn}}
		{\B{z}_n}_t = (\BA+\BP) \B{z}_n + f_n, \label{tdn1} \\[2mm]
		\B{z}_n(0) = \B{z}_n^0.\label{tdn2} 
	\end{subnumcases}
	For \eqref{tdn}, inequality \eqref{contest1} applies and we get \begin{equation}\label{contest4} 
		\|\B{z}_n\|_{C([0,T];{\BL^2(\Omega)})} \leqslant e^{\gamma T}\left(\|\B{z}_n^0\|_{{\BL^2(\Omega)}}+\int_0^T\|f_n(\tau)\|_{\BL^2(\Omega)}d\tau\right).
	\end{equation} 
	Inequality \eqref{contest3n} then follows by taking $n \to \infty.$ We now recall that $\|\B{u}\|_{\BH^1(\Omega)} \simeq \|(\BA+\BP)\B{u}\|_{\BL^2(\Omega)}$ for all $\B{u} \in \calD(\BA).$ This, along with \eqref{e232nanew2}, implies that for each $t \in [0,T]$ we have
	\begin{align*}
		\|\B{u}\|_{C([0,T],\BH^1(\Omega))}  
		& \simeq \|(\BA+\BP)\B{u}\|_{C([0,T];\BL^2(\Omega))} \\ 
		& \leqslant \|\B{u}_t\|_{C([0,T];\BL^2(\Omega))} + \|\BB\Phi + G\|_{C([0,T];\BL^2(\Omega))}
	\end{align*}
	which yields inequality \eqref{contest2}.
\end{proof}

Next, we prove some Lipschitz-type inequalities for the function $\BF.$
\begin{lemma}\label{regF} 	
	The map $\BF$ maps $\BBX_T^\BCK(\BH^1(\Omega),\BL^2(\Omega))$ to $\BBX_T(\BH^1(\Omega),\BL^2(\Omega)).$ Moreover, for each $t \in [0,T]$ and $\B{w}_1, \B{w}_2 \in  \BBX_T^\BCK(\BH^1(\Omega),\BL^2(\Omega))$ we have the following inequalities:
	\begin{equation}\label{LipF0}
		\|\BF (\B{w}_1)(t) - \BF (\B{w}_2)(t)\|_{\BL^2(\Omega )} 
		\lesssim \left\|\B{w}_1(t)- \B{w}_2(t)\right\|_{\BL^2(\Omega )},
	\end{equation} 
	\begin{equation}\label{LipF1}
		\|\BF(\B{w}_1)(t) - \BF(\B{w}_2)(t)\|_{\BH^1(\Omega)} 
		\lesssim \left\|\B{w}_1(t) - \B{w}_2(t)\right\|_{\BH^1(\Omega)}, 
	\end{equation} 
	\begin{equation}\label{LipF2}
		\|\partial_t\BF(\B{w}_1)(t) - \partial_t\BF(\B{w}_2)(t)\|_{\BL^2(\Omega)} 
		 \lesssim \left\|\B{w}_1(t) - \B{w}_2(t)\right\|_{\BH^1(\Omega)} + \left\|\partial_t\B{w}_1(t) - \partial_t\B{w}_2(t)\right\|_{\BL^2(\Omega)}.
	\end{equation} 
\end{lemma}
\begin{proof} 
	We prove the case $m = 1$ since the proof for the general $m$ is an obvious extension of it. In this case the suitable compact set $\BCK$ is of the form $\BCK = [a,b] \times [c,d]$, with $0< a < b$, $c<d$ and $a + p_\mathrm{in} \leqslant b.$
	
	 Let $\B{v} = (p,q) \in \BBX_T^\BCK(\BBH^1(\Omega),\BBL^2(\Omega))$, and let us show that $\BBF(\B{v}) \in \BBX_T(\BBH^1(\Omega),\BBL^2(\Omega))$.
	We start by showing $\BBF(\B{v}) \in C^1([0,T];{\BBL^2(\Omega)}).$ Continuity in time follows from the fact that every function involved is continuous and $p$ does not vanish in $\overline{Q}_T$. Membership in ${\BBL^2(\Omega)}$ follows by the Sobolev embedding $H^1(\Omega) \hookrightarrow L^\infty(\Omega) \hookrightarrow L^4(\Omega)$. Indeed, for each $t$, 
	\begin{align*}
		\int_\Omega |\BBF(\B{v})(t,\B{x})|^2d\B{x} 
		&= c^2\beta^2\int_\Omega \left|\dfrac{q(t,x)|q(t,x)|}{p(t,x)}\right|^2dx
		\lesssim 
		\int_{\Omega} |q(t,x)|^4dx
		\lesssim 
		\|q(t,\cdot)\|_{L^\infty(\Omega)}^4 
		<
		\infty.
	\end{align*}
	
	Now, since the function $x \mapsto x|x|$ is continuously differentiable and $p$ and $q$ are both $C^1$ in time (with values in ${\BBL^2(\Omega)}$) we have that the second component of $\BBF(\B{v})$ can be differentiated in time. Moreover, $$\dfrac{\partial}{\partial t} \dfrac{q|q|}{p} = \dfrac{2|q|q_tp- q|q|p_t}{{p}^2}$$ is again continuous in time due to continuity of all the functions involved. For membership in ${\BBL^2(\Omega)}$ we compute for each $t$
	\begin{align*}
		\int_\Omega |D_t\BBF(\B{v})(t,\B{x})|^2d\B{x} 
		&= c^2\beta^2\int_\Omega \left|\dfrac{2q_t(t,x)|q(t,x)|p(t,x) - q(t,x)|q(t,x)|p_t(t,x)}{p(t,x)}\right|^2dx \\
		& \lesssim  \int_{\Omega} |q_t(t,x)|^2|q(t,x)|^2dx + \int_\Omega|q(t,x)|^4|p_t(t,x)|^2dx \\
		& \lesssim \|q(t,\cdot)\|_{L^\infty(\Omega)}^2\|q_t(t,\cdot)\|_{L^2(\Omega)}^2 + \|q(t,\cdot)\|_{L^\infty(\Omega)}^4\|p_t(t,\cdot)\|_{L^2(\Omega)}^2 \\
		& \lesssim  \|q_t(t,\cdot)\|_{L^2(\Omega)}^2 + \|p_t(t,\cdot)\|_{L^2(\Omega)}^2< \infty.
	\end{align*}

	Therefore $\BBF(\B{v}) \in C^1([0,T],{\BBL^2(\Omega)}).$ Now, the fact that $\BBF(\B{v}) \in C([0,T],\BBH^1(\Omega))$ is a corollary of the following lemma, which is an adaptation of \cite[Corollary 8.10, p. 215]{B}, but we include a proof here for the reader's convenience.
	\begin{lemma}\label{quotfor}
		Let $\Omega = (a,b)$, $a,b \in \BR$, $a<b$, and let $u,v \in H^1(\Omega)$ be such that there exists $\rho >0$ with $v(x) \geqslant \rho$ for all $x \in \overline\Omega.$ Then $u|u|/v \in H^1(\Omega)$ and 
		\begin{equation}\label{quotientrule} 
			\left(\dfrac{u|u|}{v}\right)' = \dfrac{2|u|u'v - u|u|v'}{v^2}.
		\end{equation}
	\end{lemma}
	\begin{proof}
		Let $u_n, v_n \in C_0^1(\BR)$\footnote{Here $C_0^1(\BR)$ is the set of $C^1$ functions on $\BR$ with compact support.} such that $u_n \to u$ and $v_n \to v$ in $H^1(\Omega)$ as $n \to \infty.$ Since $v(x) \geqslant \rho$ for all $x \in \overline\Omega$ and $H^1(\Omega) \hookrightarrow C(\overline\Omega)$, we can assume without loss of generality that $v_n(x) \geqslant \rho/2$ for all $x \in \overline\Omega$, which will be the case for any sequence converging to $v$ in $H^1(\Omega)$ and $n$ large enough. 
		
		It follows by convergence that $u_n' \to u'$ and $v_n' \to v'$ in $L^2(\Omega)$ as $n \to \infty.$ Then,
		\begin{equation}\label{quot1}
			\left(\dfrac{u_n|u_n|}{v_n}\right)' = \dfrac{2|u_n|u_n'v_n-u_n|u_n|v_n'}{v_n^2} \to \dfrac{2|u|u'v - u|u|v'}{v^2} \mbox{ in } L^2(\Omega),
		\end{equation} 
		To see this, recall that $u_n \to u$ and $v_n \to v$ also in $C(\overline\Omega)$ as $n\to \infty$, then the quantities $\|u_n\|_{L^\infty(\Omega)} , \|v_n\|_{L^\infty(\Omega)} , \|u_n'\|_{L^2(\Omega)}, \|v_n'\|_{L^2(\Omega)}$ are bounded by a common quantitity for all $n$. This readily implies 
		$$\dfrac{2|u_n|u_n'v_n-u_n|u_n|v_n'}{v_n^2} \in L^2(\Omega) \mbox{ for all $n$, and} \ \dfrac{2|u|u'v - u|u|v'}{v^2} \in L^2(\Omega).$$ 
		Moreover, after some estimates one finds that there exists $C>0$ such that 
		\begin{align*}
			&
			\left\|\dfrac{2|u_n|u_n'v_n-u_n|u_n|v_n'}{v_n^2} - \dfrac{2|u|u'v - u|u|v'}{v^2}\right\|_{L^2(\Omega)} \\
			& \leqslant C\left(\|u_n-u\|_{L^\infty(\Omega)} + \|v-v_n\|_{L^\infty(\Omega)} + \|v_n'-v\|_{L^2(\Omega)} + \|u_n'-u'\|_{L^2(\Omega)}\right),
		\end{align*} 
		which implies \eqref{quot1}. This completes the proof, see \cite[Remark 4, p. 204]{B}.
	\end{proof}
	Then, for each $t$ we have
	\begin{align*}
		\int_\Omega |D_x\BBF(\B{v})(t,\B{x})|^2d\B{x} 
		&= c^2\beta^2\int_\Omega \left|\dfrac{2q_x(t,x)|q(t,x)|p(t,x) - q(t,x)|q(t,x)|p_x(t,x)}{p(t,x)}\right|^2dx \\
		& \lesssim  \int_{\Omega} |q_x(t,x)|^2|q(t,x)|^2dx + \int_\Omega|q(t,x)|^4|p_x(t,x)|^2dx \\
		& \lesssim  \|q(t,\cdot)\|_{L^\infty(\Omega)}^2\|q_x(t,\cdot)\|_{L^2(\Omega)}^2 + \|q(t,\cdot)\|_{L^\infty(\Omega)}^4\|p_x(t,\cdot)\|_{L^2(\Omega)}^2 \\
		& \lesssim \|q(t,\cdot)\|_{H^1(\Omega)}^2 + \|p(t,\cdot)\|_{H^1(\Omega)}^2 < \infty.
	\end{align*}
	We now prove the estimates. Let $\B{w}_1 = (u_1, v_1)$ and $\B{w}_2 = (u_2,v_2)$. The estimates below hold for each $t \in [0,T]$, but for simplicity we omit the time argument in the intermediate steps. We have
	\begin{align*} 
		&
		\left\|\BF(\B{w}_1)(t) - \BF(\B{w}_2)(t)\right\|_{{\BL^2(\Omega)}}^2 %\nonumber \\ & 
		= \left\|\BBF(\B{w}_1)(t) - \BBF(\B{w}_2)(t)\right\|_{{\BBL^2(\Omega)}}^2 \nonumber \\ &
		=  D^2c^2\beta^2\left\|\frac{u_1|u_1|}{v_1}-\frac{u_2|u_2|}{v_2}\right\|_{L^2(\Omega)}^2 
		=	D ^2c^2\beta^2\int_0^{L }\left|\frac{u_1 |u_1 |v_2 -u_2 |u_2 |v_1 }{v_1 v_2 }\right|^2dx  \nonumber \\ & 
		\lesssim \int_0^{L }\big|u_1 |u_1 |(v_2 -v_1 ) + (u_1 (|u_1 |-|u_2 |)+(u_1 -u_2 )|u_2 |)v_1 \big|^2dx \nonumber\\ & 
	\lesssim \|u_1 \|_{L^\infty(\Omega )}^4\|v_2 -v_1 \|_{L^2(\Omega )}^2 
		+ (\|u_1 \|_{L^\infty(\Omega )}^2+\|u_2 \|_{L^\infty(\Omega )}^2)\|v_1 \|_{L^\infty(\Omega )}^2\|u_1 -u_2 \|_{L^2(\Omega )}^2 \nonumber\\ 
		& \leqslant b^4\|v_2 -v_1 \|_{L^2(\Omega )}^2 + 2b^2d^2\|u_1 -u_2 \|_{L^2(\Omega )}^2
		\lesssim\left\|\B{w}_1(t)- \B{w}_2(t)\right\|_{\BL^2(\Omega )}.
	\end{align*} 
	Also, we have 
	\begin{align*}
		&	\left\|\partial_t \BF(\B{w}_1)(t) - \partial_t \BF(\B{w}_2)(t)\right\|_{\BL^2(\Omega)}^2 %\nonumber \\ &
		= \left\|\partial_t \BBF(\B{w}_1)(t) - \partial_t \BBF(\B{w}_2)(t)\right\|_{\BBL^2(\Omega)}^2 \nonumber \\ &
		= 	D ^2c^2\beta^2\int_0^{L }\left|\dfrac{2|{u_1 }|{u_1 }_t {v_1 } - {u_1 }|{u_1 }|{v_1 }_t}{{v_1 }^2} - \dfrac{2|{u_2 }|{u_2 }_t{v_2 } - {u_2 }|{u_2 }|{v_2 }_t}{{v_2 }^2}\right|^2dx \\&
		\lesssim\int_0^{L }\big||{u_1 }|{u_1 }_t {v_1 }{v_2 }^2 - |{u_2 }|{u_2 }_t {v_2 }{v_1 }^2\big|^2dx + \int_0^{L }\big| {u_1 }|{u_1 }|{v_1 }_t{v_2 }^2-{u_2 }|{u_2 }|{v_2 }_t{v_1 }^2\big|^2dx \\&
		\lesssim\int_0^{L }\big||{u_1 }|{v_1 }{v_2 }^2({u_1 }_t-{u_2 }_t)+{v_2 }{v_1 }{u_2 }_t({v_2 }(|{u_1 }|-|{u_2 }|) - |{u_2 }|({v_2 }-{v_1 }))\big|^2dx \\&
		+ \int_0^{L }\big| {u_1 }|{u_1 }|{v_2 }^2({v_1 }_t-{v_2 }_t)+ {v_2 }_t({u_1 }|{u_1 }|({v_2 }^2-{v_1 }^2)+({u_1 }(|{u_1 }|-|{u_2 }|)+({u_1 }-{u_2 })|{u_2 }|){v_1 }^2)\big|^2dx \nonumber\\ &
		\lesssim \|u_1\|_{L^\infty(\Omega)}^2\|v_1\|_{L^\infty(\Omega)}^2\|v_2\|_{L^\infty(\Omega)}^4\|{u_1}_t-{u_2}_t\|_{L^2(\Omega)}^2\\
		&+\|v_1\|_{L^\infty(\Omega)}^2\|v_2\|_{L^\infty(\Omega)}^2\|{u_2}_t\|_{L^2(\Omega)}^2\left(\|{v_2 }\|_{L^\infty(\Omega)}^2\|u_1-u_2\|_{L^\infty(\Omega)}^2 + \|u_2\|_{L^\infty(\Omega)}^2\|v_2-v_1\|_{L^\infty(\Omega)}^2\right) \\&
		+ \|u_1\|_{L^\infty(\Omega)}^4\|v_2\|_{L^\infty(\Omega)}^4\|{v_1 }_t-{v_2 }_t\|_{L^2(\Omega)}^2 
		+ \|{v_2 }_t\|_{L^2(\Omega)}^2\|{u_1}\|_{L^\infty(\Omega)}^4\|{v_2 }-{v_1 }\|_{L^\infty(\Omega)}^2\|v_1 + v_2\|_{L^2(\Omega)}^2 \\
		&+ \|v_1\|_{L^\infty(\Omega)}^4\left(\|u_1\|_{L^\infty(\Omega)}^2+\|u_2\|_{L^\infty(\Omega)}^2\right)\|u_1-u_2\|_{L^\infty(\Omega)}^2 \nonumber\\ &
		\lesssim b^2d^6\|{u_1}_t-{u_2}_t\|_{L^2(\Omega)}^2
		+d^4\kappa_t^2\left(d^2\|u_1-u_2\|_{L^\infty(\Omega)}^2 + b^2\|v_2-v_1\|_{L^\infty(\Omega)}^2\right) \\&
		+b^4d^4\|{v_1 }_t-{v_2 }_t\|_{L^2(\Omega)}^2 
		+ 4\kappa_t^2b^4d^2\|{v_2 }-{v_1 }\|_{L^\infty(\Omega)}^2
		+ 2d^4b^2\|u_1-u_2\|_{L^\infty(\Omega)}^2 \nonumber\\ &
		\lesssim\|{u_1}_t-{u_2}_t\|_{L^2(\Omega)}^2+\|{v_1 }_t-{v_2 }_t\|_{L^2(\Omega)}^2
		+\|u_1-u_2\|_{L^\infty(\Omega)}^2 + \|v_1-v_2\|_{L^\infty(\Omega)}^2 \\&
		\lesssim\|D_t{\B{w}_1}(t)-D_t{\B{w}_2}(t)\|_{\BL^2(\Omega)}^2
		+\|\B{w}_1(t)-\B{w}_2(t)\|_{\BL^\infty(\Omega)}^2 \\&
		\lesssim \|D_t{\B{w}_1}(t)-D_t{\B{w}_2}(t)\|_{\BL^2(\Omega)}^2+\kappa_{1,\infty}\|\B{w}_1(t)-\B{w}_2(t)\|_{\BH^1(\Omega)}^2.	
	\end{align*}
	Finally, similarly to above (by essentially exchanging time with space derivative) we have 
	\begin{align*}
		&	\left\|\partial_x \BF(\B{w}_1)(t) - \partial_x \BF(\B{w}_2)(t)\right\|_{\BL^2(\Omega)}^2 
		\lesssim \kappa_{1,\infty}\left\|\B{w}_1(t) - \B{w}_2(t)\right\|_{\BH^1(\Omega)}^2+\left\|\partial_x\B{w}_1(t) - \partial_x\B{w}_2(t)\right\|_{\BL^2(\Omega)}^2.
	\end{align*}
\end{proof}
\subsection{Semilinear analysis
}
Let $\Phi \in \BU^{ad}$ and define $\B{v}_0 = \B{v}_0(t,\B{x}) \coloneqq \B{v}_{\mathrm{e}}$ and $\B{u}_0 \coloneqq \B{v}_0 - \BB_1\Phi.$ Let $\B{u}_1 = \B{u}_1(t,\B{x})$ be the classical solution of 	
\begin{subnumcases}{\label{it1}}
	{\B{u}_1}_t = (\BA+\BP) \B{u}_1 + \BB\Phi + \BF(\B{v}_0), \label{it1a} \\[2mm]
	\B{u}_1(0) =\B{v}_e-\BB_1\Phi(0) \label{it1b} 
\end{subnumcases}
guaranteed to exist by Theorem \ref{linearu} since $\B{v}_{\mathrm{e}} - \BB_1\Phi(0) \in \calD(\BA)$ and the forcing term $\BB\Phi + \BF(\B{v}_0) \in H^1(0,T;\BL^2(\Omega)).$ 

With reference to $r>0$ in \eqref{rKK} and $c_1 \coloneqq \|\BB_1\Phi^1\|_{\BH^1(\Omega)}$, where $\Phi^1 = (1,\cdots,1)^\top \in \BR^{2m}$, we assume that $c_1\kappa_{\BU} \ll r$ (see Definition \ref{admbc2}). Let $T_1 > 0$ be the largest time $0 < T_1 \leqslant T$ such that \begin{equation}\label{t1time}\|\B{u}_1-\B{u}_0\|_{C(\overline{Q}_{T_1})} < r - c_1{\kappa_\BU}.
\end{equation} For the time being, we assume that such a time exists and continue the construction.

Let $\B{v}_1 = \B{v}_1(t,\B{x}) \coloneqq \B{u}_1 + \BB_1\Phi.$ Then, it follows from \eqref{t1time} that 
\begin{align*}
	\|\B{v}_1 - \B{v}_\mathrm{e}\|_{C(\overline{Q}_{T_1})} 
	&= \|\B{v}_1 - \B{v}_\mathrm{e}\|_{C(\overline{Q}_{T_1})} - \|\BB_1(\Phi - \Phi_\mathrm{e})\|_{C(\overline{Q}_{T_1})} +  \|\BB_1(\Phi - \Phi_\mathrm{e})\|_{C(\overline{Q}_{T_1})} \\ 
	& \leqslant \|\B{v}_1 - \B{v}_\mathrm{e} - \BB_1(\Phi - \Phi_\mathrm{e})\|_{C(\overline{Q}_{T_1})} +  \|\BB_1(\Phi - \Phi_\mathrm{e})\|_{C(\overline{Q}_{T_1})} \\ 
	&\leqslant \|\B{u}_1-\B{u}_0\|_{C(\overline{Q}_{T_1})} + c_1{\kappa_\BU} \leqslant r - c_1{\kappa_\BU} + c_1{\kappa_\BU} = r, 
\end{align*}
which means, via \eqref{rKK}, that $\B{v}_1 \in \BBX_{T_1}^\mathscr{K}(\BH^1(\Omega),\BL^2(\Omega))$, and, by Lemma \ref{regF} it follows that $\BF(\B{v}_1) \in \BBX_{T_1}(\BH^1(\Omega),\BL^2(\Omega)).$

Inductively, assuming that $\B{v}_k \in   \BBX_{T_k}^\mathscr{K}(\BH^1(\Omega),\BL^2(\Omega))$ is constructed, let $\B{u}_{k+1} = \B{u}_{k+1}(t,\B{x})$ be the classical solution of 
\begin{subnumcases}{\label{it1nnn}}
	{\B{u}_{k+1}}_t = (\BA+\BP) \B{u}_{k+1} + \BB\Phi + \BF(\B{v}_k), \label{it1annn} \\[2mm]
	\B{u}_{k+1}(0) =\B{v}_e-\BB_1\Phi(0) \label{it1bnnn} 
\end{subnumcases}
again guaranteed to exist by Theorem \ref{linearu} since $\B{v}_{\mathrm{e}} - \BB_1\Phi(0) \in \calD(\BA)$ and the forcing term $\BB\Phi + \BF(\B{v}_k) \in H^1(0,T_k;\BL^2(\Omega)).$ 

With the same assumption on $r, c_1$ and ${\kappa_\BU}$, let $T_{k+1} > 0$ be the largest time $0 < T_{k+1} \leqslant T_{k} \leqslant T$ such that \begin{equation}\label{tk1time}
	\|\B{u}_{k+1}-\B{u}_0\|_{C(\overline{Q}_{T_k})} < r - c_1{\kappa_\BU}.
\end{equation} 
Again we assume that such a time exists for the time being. We then construct $\B{v}_{k+1} = \B{v}_{k+1}(t,\B{x}) \coloneqq \B{u}_{k+1} + \BB_1\Phi.$ Then, it follows from \eqref{tk1time} that 
\begin{align*}
	\|\B{v}_{k+1} - \B{v}_\mathrm{e}\|_{C(\overline{Q}_{T_1})} 
	&\leqslant \|\B{u}_{k+1}-\B{u}_0\|_{C(\overline{Q}_{T_1})} + c_1{\kappa_\BU} \leqslant r - c_1{\kappa_\BU} + c_1{\kappa_\BU} = r, 
\end{align*}
which means, via \eqref{rKK}, that $\B{v}_{k+1} \in \BBX_{T_{k+1}}^\mathscr{K}(\BH^1(\Omega),\BL^2(\Omega))$.

The next lemma yields that the sequence of times $(T_{k})$ as constructed above has a positive lower bound.

\begin{lemma}\label{timelower} 
	There exist $M, K, T_\ast > 0$ such that the sequences of solutions $(\B{v}_{k})$ and $(\B{u}_k)$ constructed above are such that \begin{itemize}
		\item[\bf (a)] $\|\B{v}_k - \B{v}_\mathrm{e}\|_{C(\overline{Q}_{T_\ast})} \leqslant r, $
		\item[\bf (b)] $\|{\B{u}_k}-\B{u}_0\|_{C([0,T_\ast];\BH^1(\Omega))} \leqslant K,$
		\item[\bf (c)] $\|{\B{u}_k}_t\|_{C([0,T_\ast];\BL^2(\Omega))} \leqslant M.$
	\end{itemize} 
\end{lemma}
\begin{proof}
	The proof is by induction. Recall that $T$ is a \emph{fixed} large time. 

	We start by showing the base step ($k = 1$). To this end, notice that, since $\B{v}_\mathrm{e}$ is a steady state solution for the Euler system, the function $\B{z} \coloneqq \B{u}_1 - \B{u}_0 \in \BBX_{T}(\calD(\BA),\BL^2(\Omega))$ solves the abstract Cauchy problem \begin{subnumcases}{\label{it2}}
	\B{z}_t = (\BA+\BP) \B{z} + \BB(\Phi-\Phi_\mathrm{e}), \label{it2a} \\[2mm]
	\B{z}(0) = 0. \label{it2b} 
	\end{subnumcases} It then follows from \eqref{contest1} that 
	\begin{align}\label{contest1n} 
		\|\B{z}\|_{C([0,T];{\BL^2(\Omega)})} 
		&\leqslant e^{\gamma T}\int_0^T\|\BB(\Phi-\Phi_\mathrm{e})(\tau)\|_{\BL^2(\Omega)}d\tau \nonumber \\ 
		&\leqslant e^{\gamma T}T\|\BB(\Phi-\Phi_\mathrm{e})\|_{C([0,T];\BL^2(\Omega))} \lesssim e^{\gamma T}T,
	\end{align} 
	and from \eqref{contest3n} it follows that  
	\begin{align}\label{contest3n1} 
		\|\B{z}_t\|_{C([0,T];{\BL^2(\Omega)})} 
		&\leqslant e^{\gamma T}\int_0^T\|D_t\BB(\Phi-\Phi_\mathrm{e})(\tau)\|_{\BL^2(\Omega)}d\tau \nonumber \\
		&  \leqslant e^{\gamma T}\|D_t \BB(\Phi-\Phi_\mathrm{e})\|_{L^1(0,T;\BL^2(\Omega))} 
		\lesssim e^{\gamma T},
	\end{align} and from \eqref{contest2} that 
	
	\begin{equation}\label{contest21n} 
		\|\B{z}\|_{C([0,T];\BH^1(\Omega))} \lesssim e^{\gamma T}\eta + \|\BB(\Phi-\Phi_\mathrm{e})\|_{H^1(0,T;\BL^2(\Omega))} \lesssim %\eta
		e^{\gamma T}.
	\end{equation}
	Now, it follows from \eqref{emb_epsinter} that for $0 < \varepsilon \ll 1/2$ we have 
	\begin{align}
		\label{inter1} \|\B{z}\|_{C([0,T];\BH^{1/2 + \varepsilon}(\Omega))} 
		&\lesssim \|\B{z}\|_{C([0,T];\BH^1(\Omega))}^{1/2+\varepsilon}\|\B{z}\|_{C([0,T];\BL^2(\Omega))}^{1/2-\varepsilon} \lesssim e^{\gamma T}T^{1/2-\varepsilon}.
	\end{align}

	Whence, from \eqref{emb_epsinf} we obtain that 
	\begin{equation}\label{lemmaa} 
		\|\B{z}\|_{C(\overline{Q}_T)} \lesssim \|\B{z}\|_{C([0,T];\BH^{1/2 + \varepsilon}(\Omega))} \lesssim e^{\gamma T}T^{1/2-\varepsilon} \leqslant r - c_1{\kappa_\BU},
	\end{equation} 
	for $T$ small. From here we get $T_1.$ Since $T_1 \leqslant T$, from \eqref{contest21n} we get $K$ and from \eqref{contest3n1} we get $M$. This proves the step $k = 1$ of the induction argument. 

	We now assume that $\B{v}_k$ and $\B{u}_k$ satisfy (a), (b) and (c) in the statement for appropriate choices of $K, M$ and $T_\ast.$ 

	The function $\B{z}_{k+1} \coloneqq \B{u}_{k+1} - \B{u}_0 \in \BBX_{T_\ast}(\calD(\BA),\BL^2(\Omega))$ solves the abstract Cauchy problem \begin{subnumcases}{\label{it3}}
		{\B{z}_{k+1}}_t = (\BA+\BP) {\B{z}_{k+1}} + \BB(\Phi-\Phi_\mathrm{e}) + \BF(\B{v}_k)-\BF(\B{v}_{\mathrm{e}}), \label{it3a} \\[2mm]
		\B{z}_{k+1}(0) = 0. \label{it3b} 
	\end{subnumcases} 
	It follows from \eqref{contest1} and \eqref{LipF0} that	
	\begin{align}\label{contest1nn} 
		\|\B{z}_{k+1}\|_{C([0,T_\ast];{\BL^2(\Omega)})} 
		& \leqslant e^{\gamma T_\ast}\left(\int_0^{T_\ast}\|\BB(\Phi-\Phi_\mathrm{e})(\tau)\|_{\BL^2(\Omega)}d\tau  + \int_0^{T_\ast}\|\BF(\B{v}_k)(\tau)-\BF(\B{v}_{\mathrm{e}})(\tau)\|_{\BL^2(\Omega)}d\tau\right)\nonumber \\ 
		&\lesssim e^{\gamma T_\ast}T_\ast\left(\|\BB(\Phi-\Phi_\mathrm{e})\|_{C([0,T_\ast];\BL^2(\Omega))}+\|\B{v}_k-\B{v}_{\mathrm{e}}\|_{C([0,T_\ast];\BL^2(\Omega))}\right) \nonumber \\ 
		&\lesssim e^{\gamma T_\ast}T_\ast\left(1+\|\B{v}_k-\B{v}_{\mathrm{e}}\|_{C(\overline{Q}_{T_\ast})}\right) %\nonumber \\ &
		\lesssim e^{\gamma T}T_\ast.
	\end{align}  
	Then, from \eqref{contest3n} and \eqref{LipF2} we infer
 	\begin{align}\label{contest3nn1} 
		&\|{\B{z}_{k+1}}_t\|_{C([0,T_\ast];{\BL^2(\Omega)})} \nonumber \\ 
		&\leqslant e^{\gamma T_\ast}\left(\int_0^{T_\ast}\|D_t\BB(\Phi-\Phi_\mathrm{e})(\tau)\|_{\BL^2(\Omega)}d\tau + \int_0^{T_\ast}\|D_t(\BF(\B{v}_k)-\BF(\B{v}_{\mathrm{e}}))(\tau)\|_{\BL^2(\Omega)}d\tau\right)\nonumber \\ 
		& \lesssim e^{\gamma T_\ast}+ e^{\gamma T_\ast}\int_0^{T_\ast}\left[\|\B{v}_k(\tau)-\B{v}_{\mathrm{e}}\|_{\BH^{1}(\Omega)} + \|\partial_t\B{v}_k(\tau)\|_{\BL^2(\Omega)}\right]d\tau \nonumber \\ 
		& \lesssim (1+T_\ast)e^{\gamma T_\ast}+ e^{\gamma T_\ast}\int_0^{T_\ast}\left[\|\B{u}_k(\tau)-\B{u}_0\|_{\BH^1(\Omega)} + \|\partial_t\B{u}_k(\tau)\|_{\BL^2(\Omega)}\right]d\tau \nonumber \\ 
		&  \lesssim (1+T_\ast)e^{\gamma T_\ast} \leqslant (1+T)e^{\gamma T},
	\end{align} 
	since $T_\ast \leqslant T.$ Finally from \eqref{contest2}, \eqref{contest3nn1} and \eqref{LipF0} it follows that
	\begin{align}\label{contest21} 
		\|\B{z}_{k+1}\|_{C([0,T_\ast];\BH^1(\Omega))} 
		&\lesssim (1+T)e^{\gamma T} + \|\BB(\Phi-\Phi_\mathrm{e})\|_{C([0,T_\ast];\BL^2(\Omega))} + \|\BF(\B{v}_k)-\BF(\B{v}_{\mathrm{e}})\|_{C([0,T_\ast];\BL^2(\Omega))} \nonumber \\ 
		&\lesssim (1+T)e^{\gamma T} + \|\BB(\Phi-\Phi_\mathrm{e})\|_{H^1(0,T_\ast;\BL^2(\Omega))} + \|\B{v}_k-\B{v}_{\mathrm{e}}\|_{C([0,T_\ast];\BL^2(\Omega))} \nonumber \\ 
		&\lesssim (1+T)e^{\gamma T}.
	\end{align} 
	We finish the argument by again using the interpolation inequality \eqref{emb_epsinter}, i.e.,  for $0 < \varepsilon \ll 1/2$ we have 
	\begin{align}
		\label{inter1n} \|\B{z}_{k+1}\|_{C(\overline{Q}_T)} 
		&\lesssim \|\B{z}_{k+1}\|_{C([0,T_\ast];\BH^{1/2 + \varepsilon}(\Omega))} \nonumber \\ 
		&\lesssim \|\B{z}_{k+1}\|_{C([0,T_\ast];\BH^1(\Omega))}^{1/2+\varepsilon}\|\B{z}_{k+1}\|_{C([0,T_\ast];\BL^2(\Omega))}^{1/2-\varepsilon} \nonumber \\ 
		&\lesssim [(1+T)e^{\gamma T}]^{1/2+\varepsilon}(e^{\gamma T}T_\ast)^{1/2-\varepsilon} \lesssim {T_\ast}^{1/2-\varepsilon} \leqslant r - c_1{\kappa_\BU},
	\end{align} 
	for $T_\ast$ small. Since the last constant in \eqref{inter1n} does \emph{not} depend on $T_\ast$, the proof follows.
\end{proof}

The next lemma is crucial to guarantee that the constructed sequence has a limit.

\begin{lemma}[\bf Contraction] 
	There exists $T_{c}$, possibly smaller than $T_\ast$, and $0 < \delta < 1$ such that the sequence $(\B{u}_k)$ satisfies 
	\begin{equation}\label{contract} 
		\|\B{u}_{k+1}-\B{u}_k\|_{C([0,T_c];\BL^2(\Omega))} \leqslant \delta \|\B{u}_{k}-\B{u}_{k-1}\|_{C([0,T_c];\BL^2(\Omega))}. 
	\end{equation}
\end{lemma} 
\begin{proof}
	The proof is rather straightforward. Notice that if $\B{w}_{k+1} \coloneqq \B{u}_{k+1}-\B{u}_k$, then it follows by \eqref{it3} that $\B{w}_{k+1}$ satisfies the abstract Cauchy problem 
	\begin{subnumcases}{\label{it4}}
		{\B{w}_{k+1}}_t = (\BA+\BP) {\B{w}_{k+1}} + \BF(\B{v}_k)-\BF(\B{v}_{k-1}), \label{it4a} \\[2mm]
		\B{z}_{k+1}(0) = 0. \label{it4b} 
	\end{subnumcases} 
	Then it follows from \eqref{contest1} and \eqref{LipF0} that	
	\begin{align}\label{contest1nnn} 
		\|\B{w}_{k+1}\|_{C([0,T_\ast];{\BL^2(\Omega)})} 
		&\leqslant e^{\gamma T_\ast} \int_0^{T_\ast}\|\BF(\B{v}_k)(\tau)-\BF(\B{v}_{k-1})(\tau)\|_{\BL^2(\Omega)}d\tau \nonumber \\ 
		&\lesssim e^{\gamma T_\ast} \int_0^{T_\ast}\|\B{v}_k(\tau)-\B{v}_{k-1}(\tau)\|_{\BL^2(\Omega)}d\tau\nonumber \\ 
		&= e^{\gamma T_\ast} \int_0^{T_\ast}\|\B{u}_k(\tau)-\B{u}_{k-1}(\tau)\|_{\BL^2(\Omega)}d\tau 
		\lesssim T_\ast\|\B{w}_k\|_{C([0,T^\ast];\BL^2(\Omega))}
	\end{align} whereby the proof is complete for suitably chosen $T_c$.
\end{proof}

As a corollary of the previous lemma, we obtain convergence of $(\B{u}_k)$. 
\begin{corollary}
	There exists $\B{u} \in \BX_{T_c}(\calD(\BA),\BL^2(\Omega))$ such that 
	\begin{equation}\label{conve1} 
		\B{u}_k \to \B{u} \ \mbox{in} \ \BX_{T_c}(\BH^1(\Omega),\BL^2(\Omega))\quad\text{as }k\to\infty.
	\end{equation}
\end{corollary}
\begin{proof}
	Summing up inequality \eqref{contract} we obtain, after rearranging, 
	\begin{equation}\label{contract2} 
		\sum\limits_{k=1}^{\infty}\|\B{u}_{k+1}-\B{u}_k\|_{C([0,T_c];\BL^2(\Omega))} \leqslant \dfrac{\delta}{1-\delta}  \|\B{u}_{1}-\B{u}_{0}\|_{C([0,T_c];\BL^2(\Omega))} < \infty.
	\end{equation} Moreover, by using $\B{w}_{k+1} = \B{u}_{k+1}-\B{u}_k$ in \eqref{it4} for estimates, it follows from \eqref{contest3n} that 
	\begin{align}\label{contest32n} 
		\|{\B{w}_{k+1}}_t\|_{C([0,T_c];{\BL^2(\Omega)})} 
		&\leqslant e^{\gamma T_c}\int_0^{T_c}\|D_t(\BF(\B{v}_k)-\BF(\B{v}_{k-1}))(
		\tau)\|_{\BL^2(\Omega)}d\tau \nonumber \\ 
		&\lesssim T_c\left[\left\|\B{w}_{k}\right\|_{C([0,T_c];\BH^1(\Omega))}+\left\|\partial_t\B{w}_{k}\right\|_{C([0,T_c];\BL^2(\Omega))}\right]
		\end{align}  and from \eqref{contest2}
	\begin{equation}\label{contest2n} 
		\|{\B{w}_{k+1}}\|_{C([0,T];\BH^1(\Omega))}\lesssim T_c\left[\left\|\B{w}_{k}\right\|_{C([0,T_c];\BH^1(\Omega))}+\left\|\partial_t\B{w}_{k}\right\|_{C([0,T_c];\BL^2(\Omega))}\right] + \|\B{w}_{k}\|_{C([0,T];\BL^2(\Omega))}.
	\end{equation} 
	Then adding \eqref{contest32n} and \eqref{contest2n} and summing up over $k$ we have, after rearranging and using \eqref{contract2},
	\begin{align}\label{add1} 
		&\sum\limits_{k=1}^\infty \|{\B{u}_{k+1}-\B{u}_k}_t\|_{\BX_{T_c}(\BH^1(\Omega),\BL^2(\Omega))} \nonumber \\
		&\lesssim \sum\limits_{k=1}^\infty\|\B{u}_{k+1}-\B{u}_k\|_{C([0,T];\BL^2(\Omega))} +  \dfrac{\delta}{1-\delta}  \|\B{u}_{1}-\B{u}_{0}\|_{C([0,T_c];\BL^2(\Omega))} \nonumber \\ 
		& +\dfrac{T_c}{1-T_c}\left[\|{\B{u}_{1}-\B{u}_0}_t\|_{C([0,T_c];{\BL^2(\Omega)})} + \|{\B{u}_{1}-\B{u}_0}\|_{C([0,T];\BH^1(\Omega))}\right]  < \infty,
	\end{align} 
	for appropriate (smaller, if necessary) $T_c.$ From this it follows that there exists $\B{u} \in \BBX_{T_c}(\BH^1(\Omega),\BL^2(\Omega))$ such that $\B{u}_k \to \B{u}$ in $\BX_{T_c}(\BH^1(\Omega),\BL^2(\Omega))$. It follows from the construction that $\B{u} \in \calD(\BA)$ and solves \begin{subnumcases}{\label{it1nn}}
		{\B{u}}_t = (\BA+\BP) \B{u} + \BB\Phi + \BF(\B{v}), \label{it1nna} \\[2mm]
		\B{u}(0) =\B{v}_e-\BB_1\Phi(0) \label{it1nnb} 
	\end{subnumcases} with $\B{v} = \B{u} + \BB_1\Phi.$ This completes the proof.
\end{proof}

In order to complete the proof of Theorem \ref{mainwpp}, we recall that $\B{v}_k = \B{u}_k + \BB_1\Phi$. By making $k \to \infty$ we get $\B{v}_k \to \B{u} + \BB_1\Phi := \B{v}(\Phi)$ in $\BBX_{T_c}(\BH^1(\Omega),\BL^2(\Omega))$.  Since, from the construction, $\B{v}_k \in \BX_{T_c}^\mathscr{K}(\BH^1(\Omega),\BL^2(\Omega))$ for each $k$, so does $\B{v}(\Phi).$ This justifies the definition of the control-to-state map $\BBS: \BU^{ad} \to \BBX_{T_c}^\BCK(\BH^1(\Omega),\BL^2(\Omega))$ as in \eqref{ctsmap} as $\BS(\Phi) = \B{v}(\Phi).$

\section{Optimal control problem}\label{control}
\subsection{Existence of optimal controls}\label{existence_control}

Rather then with $\BBS$ we work here with $\BS: \BU_0^\eta \subset \BU_0 \to \BBX_{T_c}^\BK\left(\BH^1(\Omega),\BL^2(\Omega)\right)$, with
$\BS(\Phi) := \BBS(\phi(\Phi)) = \BBS(\Phi + \Phi^\mathrm{e})$; see Remark \ref{U0eta}. Thus, we study the well-posedness of the reduced minimization problem:
\begin{equation}
	\begin{aligned}
		\min\limits \ &  J(\Phi) \coloneqq \hat J\left(\Phi,\BS(\Phi)\right) =  \dfrac{1}{2}\|\BS(\Phi)-\B{v}_d\|_{L^2(0,T;\BL^2(\Omega))}^2 + \dfrac{\sigma}{2}\|\Phi\|_{\BU}^2\\
		\mathrm{s.t.} \ \  &\Phi \in \BU_0^\eta \ \mathrm{and} \ \BS(\Phi) \in \BV_{T_c,\BK}^\mathrm{ad}.\\
	\end{aligned}
	\tag{${rP}_{T}^\BK$}\label{minp4n}
\end{equation}
\begin{proposition}[\bf Existence of optimal control]\label{existoptimal3}
	Assume that $\BS:\BU_0^\eta \to L^2(0,T_c;\BL^2(\Omega))$ is {weak--to--strong} continuous, i.e., given a sequence $(\Phi_n)$ in $\BU_0^\eta$ 
	\begin{equation}\label{wtos}
	\Phi_n \rightharpoonup \Phi \ \mbox{\rm in} \ \BU \Longrightarrow \BS(\Phi_n) \to \BS(\Phi) \ \mbox{\rm in} \ L^2(0,T_c;\BL^2(\Omega)).
	\end{equation}
	Then, \eqref{minp4} admits a solution.
\end{proposition}
\begin{proof}
	First, notice that the feasible set of \eqref{minp4n}%, i.e., the set of $\Phi \in \BU_0^\eta$ such that $\BS(\Phi) \in \BV_{T_c,\BK}^\mathrm{ad}$, 
  is nonempty since it contains zero. This along with the non-negativity of $J$ allow us to define 
	\begin{equation}\label{inf2} 
		d := \inf\{J(\Phi); \Phi \ \mbox{is feasible}\} \geqslant 0.
	\end{equation} We want to show that $d = J (\Phi^\ast)$ for some feasible $\Phi^\ast$.
	
	By properties of the infimum, there exists a sequence $(\Phi_n)$ of feasible controls such that $J(\Phi_n) \to d$ in $\BR$ as $n \to \infty.$ Since $\BU_0^\eta$ is weakly closed, there exists $\Phi^\ast \in \BU_0^\eta$ and a (non--relabeled) subsequence $(\Phi_n)$ such that $\Phi_n \rightharpoonup \Phi^\ast$ in $\BU.$ By {weak--to--strong} continuity of $\BS:\BU_0^\eta \to L^2(0,T_c;\BL^2(\Omega))$ we have $\BS(\Phi_n) \to \BS(\Phi^\ast).$ Moreover, since $\Phi_n$ is feasible for all $n$ we have $\BS(\Phi_n) \in \BV_{T_c,\BK}^\mathrm{ad}$ for all $n$ and since $\BK$ is compact, it follows that $\BS(\Phi^\ast) \in \BV_{T_c,\BK}^\mathrm{ad}$. Hence, $\Phi^\ast$ is feasible.  
	
	Now, {weak--to--strong} continuity of $\BS$ implies weak--to--strong continuity of the first summand of $J$ and the second summand is weakly lower semi-continuous. Therefore, we have 
	\begin{align*}
		d &= \liminf_{n \to \infty} J(\Phi_n) 
		= \liminf_{n\to \infty} \left[\dfrac{1}{2}\|\BS(\Phi_n)-\B{v}_d\|_{L^2(0,T;\BL^2(\Omega))}^2 + \dfrac{\sigma}{2}\|\Phi_n\|_{\BU}^2 \right] \\
		& = \dfrac{1}{2} \lim_{n\to \infty} \|\BS(\Phi_n)-\B{v}_d\|_{L^2(0,T;\BL^2(\Omega))}^2 + \dfrac{\sigma}{2} \liminf_{n\to \infty}  \|\Phi_n\|_{\BU}^2 
		 \geqslant J(\Phi^\ast) \geqslant d,
	\end{align*} and then $\Phi^\ast$ is an optimal solution.
\end{proof}
We now show that $\BS:\BU_0^\eta \to L^2(0,T_c;\BL^2(\Omega))$ is actually weak--to--strong continuous. For establishing this fact we need a series of lemmas. We start by showing a Lipschitz property.
\begin{lemma}\label{Lip}
	By decreasing $T_c$ (if necessary), the map $\BS: \BU_0^\eta \subset \BU_0 \to \BBX_{T_c}^\BCK(\BH^1(\Omega),\BL^2(\Omega))$ is Lipschitz continuous. 
\end{lemma}
\begin{proof}
	Let $\tilde \Phi, \tilde \Psi \in \BU_0^\eta$ and define $\Phi = \tilde \Phi + \Phi^\mathrm{e}$, $\Psi = \tilde \Psi + \Phi^\mathrm{e}$. By setting $\B{u}_{(\cdot)} := \BBS(\cdot) - \BB_1(\cdot)$ and in view of the construction of $\BBS$, it follows that $\B{u}_\Phi - \B{u}_\Psi$ is a classical solution of the Cauchy problem \begin{subnumcases}{\label{it1c}}
		{(\B{u}_\Phi - \B{u}_\Psi)}_t = (\BA+\BP)(\B{u}_\Phi - \B{u}_\Psi) + \BB(\Phi-\Psi) + \BF(\BBS(\Phi)) - \BF(\BBS(\Psi)), \label{it1ca} \\[2mm]
		(\B{u}_\Phi - \B{u}_\Psi)(0) =0. \label{it1cb} 
	\end{subnumcases} By \eqref{contest1} it follows that
	\begin{align*}
		\|\B{u}_\Phi - \B{u}_\Psi\|_{C(0,T_c);\BL^2(\Omega)} 
		&\leqslant e^{\gamma T_c} \int_0^{T_c} \left\|\BB(\Phi - \Psi)(\tau) + \BF(\BBS(\Phi))(\tau)-\BF(\BBS(\Psi))(\tau)\right\|_{\BL^2(\Omega)}d\tau \\
		& \lesssim \left\|\BB(\Phi - \Psi)\right\|_{L^1(0,T;\BL^2(\Omega))} + \int_0^{T_c}\left\|\BF(\BBS(\Phi))(\tau)-\BF(\BBS(\Psi))(\tau)\right\|_{\BL^2(\Omega)}d\tau.
	\end{align*}  From \eqref{LipF0} we obtain 
	\begin{align*}
		\|\B{u}_\Phi - \B{u}_\Psi\|_{C(0,T_c);\BL^2(\Omega)} 
		& \lesssim \left\|\BB(\Phi - \Psi)\right\|_{L^1(0,T;\BL^2(\Omega))} + T_c\left\|\BBS(\Phi)-\BBS(\Psi)\right\|_{C([0,T_c];\BL^2(\Omega))}.
	\end{align*}
	Lemma \ref{regB} then yields
	\begin{align}\label{LipL2}
		\|\BBS(\Phi) - \BBS(\Psi)\|_{C([0,T];\BL^2(\Omega))} 
		&\lesssim  \dfrac{1}{1-T_c}\|\Phi-\Psi\|_{\BU}.
	\end{align}
	Now, by \eqref{contest3n} it follows that
	\begin{align*}
		\|D_t(\B{u}_\Phi - \B{u}_\Psi)\|_{C(0,T_c);\BL^2(\Omega)} 
		&\leqslant e^{\gamma T_c} \int_0^{T_c} \left\|D_t\BB(\Phi - \Psi)(\tau) + D_t\BF(\BBS(\Phi))(\tau)-D_t\BF(\BBS(\Psi))(\tau)\right\|_{\BL^2(\Omega)}d\tau \\
		& \lesssim \left\|D_t\BB(\Phi - \Psi)\right\|_{L^1(0,T;\BL^2(\Omega))} + \int_0^{T_c}\left\|D_t\left[\BF(\BBS(\Phi))-\BF(\BBS(\Psi))\right](\tau)\right\|_{\BL^2(\Omega)}d\tau,
	\end{align*}  
	and then from \eqref{LipF2} we infer 
	\begin{align}\label{stlevell2}
		&\|D_t(\BBS(\Phi)-\BBS(\Psi))\|_{C(0,T_c;\BL^2(\Omega))} \nonumber \\
		& \lesssim \left\|\BB D_t(\Phi - \Psi)\right\|_{L^1(0,T;\BL^2(\Omega))} + \left\|\BB_t D_t(\Phi - \Psi)\right\|_{C([0,T];\BL^2(\Omega))} \nonumber \\
		&+ T_c\left[\left\|\BBS(\Phi)-\BBS(\Psi)\right\|_{C([0,T_c];\BH^1(\Omega))} + \left\|D_t(\BBS(\Phi)-\BBS(\Psi))\right\|_{C([0,T_c];\BL^2(\Omega))}\right].
	\end{align}
	Using \eqref{contest2} along with \eqref{LipL2} we have 
	\begin{align}\label{slevelh1}
		&\|\BBS(\Phi) - \BBS(\Psi)\|_{C(0,T_c;\BH^1(\Omega))} \nonumber \\
		& \lesssim \left\|\BB D_t(\Phi - \Psi)\right\|_{L^1(0,T;\BL^2(\Omega))} + \left\|\BB(\Phi - \Psi)\right\|_{C([0,T_c];\BL^2(\Omega))} \nonumber\\
		&+\left\|\BB(\Phi - \Psi)\right\|_{C([0,T];\BH^1(\Omega))}+ \dfrac{1}{1-T_c}\|\Phi-\Psi\|_{\BU} \nonumber \\
		&+ T_c\left[\left\|\BBS(\Phi)-\BBS(\Psi)\right\|_{C([0,T_c];\BH^1(\Omega))} + \left\|D_t(\BBS(\Phi)-\BBS(\Psi))\right\|_{C([0,T_c];\BL^2(\Omega))}\right].
	\end{align}
	Therefore, adding \eqref{LipL2}, \eqref{stlevell2} and \eqref{slevelh1} and using Lemma \ref{regB} where necessary, we have
	\begin{align}\label{LipXH1L2}
		\|\BBS(\Phi) - \BBS(\Psi)\|_{\BBX_{T_c}(\BH^1(\Omega),\BL^2(\Omega))}  
		&\lesssim  \dfrac{2-T_c}{(1-T_c)^2}\|\Phi-\Psi\|_{\BU} \lesssim \|\Phi-\Psi\|_\BU,
	\end{align}
	and the proof is finished by noticing that $\BBS(\Phi) = \BS(\tilde\Phi)$, $\BBS(\Psi) = \BS(\tilde\Psi)$ and $\Phi - \Psi = \tilde\Phi - \tilde\Psi.$
\end{proof} 
In all the other results presented in this section, we work with a (possibly smaller) $T_c$ such that $\BS: \BU_0^\eta \to \BBX_{T_c}^\BCK(\BH^1(\Omega),\BL^2(\Omega))$ is Lipschitz continuous. An immediate corollary of the previous lemma is the \emph{strong--to--strong} continuity of the map $\BS: \BU_0^\eta \to L^2(0,T_c; \BL^2(\Omega))$. The next \emph{essential} lemma shows that $\BS$ is weakly closed, i.e., given a sequence $(\Phi_n) \in \BU_0^\eta$  one has
\begin{equation}\label{weaklyclosed}
	\Phi_n \rightharpoonup \Phi \ \mbox{\rm in} \ \BU \ \mbox{\rm and} \ \BS(\Phi_n) \rightharpoonup \B{v} \ \mbox{\rm in} \ L^2(0,T_c; \BL^2(\Omega)) \Longrightarrow \Phi \in \BU_0^\eta \ \mbox{\rm and} \ \BS(\Phi) = \B{v}.
\end{equation}
In order to prove weak--closedness of $\BS$ we need to weaken the notion of solution for the semilinear Euler system. Below we introduce the notion of weak solutions for the semilinear abstract problem \eqref{abs1}. To that end, let $\mathfrak{X}_T^\BF \subset C([0,T];\BL^2(\Omega))$ be the \emph{largest} set with the following properties: 
\begin{itemize}
	\item[(a)] For all $\B{v} \in \mathfrak{X}_T^\BF$ we have that $\BF(\B{v})$ is well defined and $\BF(\B{v}) \in C([0,T];\BL^2(\Omega))$; and
	\item[(b)] the map $\BF$ is Lipschitz in $C([0,T];\BL^2(\Omega))$, i.e., there exists $L>0$ such that $$\| \BF(\B{v}_1)- \BF(\B{v}_2)\|_{C([0,T];\BL^2(\Omega))} \leqslant L\|\B{v}_1-\B{v}_2\|_{C([0,T];\BL^2(\Omega))}.$$
\end{itemize}
Of course, the set $\mathfrak{X}_T^\BF$ is nonempty. Indeed, it was proven in Lemma \ref{regF} that, for example, $$\BBX_T^\BCK(\BH^1(\Omega),\BL^2(\Omega)) \subset \mathfrak{X}_T^\BF$$ for any suitable set $\BCK$. 
\begin{definition}\label{weak}
	Let $\Phi \in \BU^{\mathrm{ad}}$. We say that $\B{v}: \Omega \times (0,T) \to \BR^{2m}$ is a {\bf weak solution} of the semilinear Euler system if 
	\begin{itemize}
		\item[\bf (i)] $\B{v} \in \mathfrak{X}_T^\BF$;
		\item[\bf (ii)] $(\B{v}(0,\cdot) - \BB_1\Phi(0),\B{z})_{\BL^2(\Omega)} = (\B{v}_\mathrm{e}-\BB_1\Phi^\mathrm{e},\B{z})_{\BL^2(\Omega)}$ for all $\B{z} \in \calD(\BA)$;
		\item[\bf (iii)] the function $t \mapsto (\B{v}(t,\cdot),\B{z})$ belongs to $[H^1(0,T)]^{2m}$ for all $\B{z} \in \calD(\BA)$; and
		\item[\bf (iv)] for all $\B{z} \in \calD(\BA)$ and $t \in [0,T]$ we have 
		\begin{equation}\label{weakform}
			\dfrac{d}{dt}(\B{u}(t,\cdot),\B{z})_{\BL^2(\Omega)} = (\B{u}(t,\cdot),(-\BA+\BP^\ast)\B{z})_{\BL^2(\Omega)} + (\BB\Phi(t),\B{z})_{\BL^2(\Omega)} + (\BF(\B{v})(t,\cdot),\B{z})_{\BL^2(\Omega)}
		\end{equation} where $\B{u} = \B{v} - \BB_1\Phi.$
	\end{itemize}
\end{definition}

It is not difficult to see that a classical solution is \emph{weak} and that a smooth \emph{weak} solution is classical, so the definition is consistent. Since we already know that for reasonably chosen $T_c$, local classical solution exists, it is clear that a weak solution also exists. When it comes to weak solutions, however, since we are making the solution space \emph{bigger}, we might lose \emph{uniqueness}, which is a key ingredient in the proof of weak--closedness of $\BS$. The set $\mathfrak{X}_T^\BF$ here represents a set with conditions under which weak solutions are well defined and \emph{unique}. We do not know whether it is the maximal set where weak solutions are unique, but it is big enough for our purposes.
\begin{theorem}
	The semilinear Euler system admits at {most} one weak solution in $\mathfrak{X}_T^\BF.$
\end{theorem}
\begin{proof}
	Assume $\B{v}_1, \B{v}_2 \in \mathfrak{X}_T^\BF$ are both weak solutions of \eqref{abs1} associated to the control $\Phi \in \BU^\mathrm{ad}$. Of course they agree at $t = 0$. Define $\B{w} \coloneqq \B{v}_1 - \B{v}_2$ and notice that, by definition, $\B{w}$ is such that for all $\B{z} \in \calD(\BA)$ and $t \in [0,T]$ we have 
	\begin{equation}\label{weakform2}
		\dfrac{d}{dt}(\B{w}(t,\cdot),\B{z})_{\BL^2(\Omega)} = (\B{w}(t,\cdot),(-\BA+\BP^\ast)\B{z})_{\BL^2(\Omega)} + (\BF(\B{v}_1)(t,\cdot)-\BF(\B{v}_2)(t,\cdot),\B{z})_{\BL^2(\Omega)}.
	\end{equation} Let $0 <\tau \leqslant T$ and define $\lambda(\tau,t) \coloneqq \BT^\ast(\tau-t)\B{z}$ for $0 \leqslant t \leqslant \tau$ where $\BT^\ast$ is the dual semigroup, generated by $(\BA+\BP)^\ast = -\BA + \BP^\ast.$ It follows by semigroup theory that
	\begin{align*}
		\dfrac{d}{dt}(\lambda(\tau,t),\B{w})_{\BL^2(\Omega)} 
		&= (-(\BA+\BP)^\ast\lambda(\tau,t),\B{w})_{\BL^2(\Omega)} + (\lambda(\tau,t),(\BA+\BP)\B{w} + \BF(\B{v}_1)-\BF(\B{v}_2))_{\BL^2(\Omega)} \\
		&=(\lambda(\tau,t),\BF(\B{v}_1)-\BF(\B{v}_2))_{\BL^2(\Omega)}.
	\end{align*} 
	Then, after integrating in time on $(0,\tau)$ and recalling that $\B{w}(0) = 0$ we have
	\begin{align*}
		(\lambda(\tau,\tau),\B{w})_{\BL^2(\Omega)} =  \int_0^\tau(\B{z},\BT(\tau-t)(\BF(\B{v}_1)-\BF(\B{v}_2))(t))_{\BL^2(\Omega)}dt,
	\end{align*}
	which simplifies to
	\begin{align*}
		(\B{z},\B{w})_{\BL^2(\Omega)} =  \left(\B{z},\int_0^\tau\BT(\tau-t)(\BF(\B{v}_1)-\BF(\B{v}_2))(t)dt\right)_{\BL^2(\Omega)},
	\end{align*}
	and since $\calD(\BA)$ is dense in $\BL^2(\Omega)$ it follows that $$\B{w} = \int_0^\tau\BT(\tau-t)(\BF(\B{v}_1)-\BF(\B{v}_2))(t)dt.$$
	From this we estimate
	\begin{align*}
		\|\B{v}_1(\tau,\cdot)-\B{v}_2(\tau,\cdot)\|_{\BL^2(\Omega)} \leqslant CL \int_0^\tau e^{\gamma (\tau-t)}\|\B{v}_1(t,\cdot)-\B{v}_2(t,\cdot)\|_{\BL^2(\Omega)}dt,
	\end{align*}
	whereby the result follows from Gr\"ownwall's inequality.
\end{proof}
We are now ready to show weak--closedness of $\BS.$
\begin{lemma}
	The map $\BS: \BU_0^\eta \to L^2(0,T_c; \BL^2(\Omega))$ is weakly closed.
\end{lemma}
\begin{proof}
	Let $(\Phi_n)$ be a sequence in $\BU_0^\eta$ such that $\Phi_n \rightharpoonup \Phi $ in $\BU$ and $\BS(\Phi_n) \rightharpoonup \B{v}$ in $L^2(0,T_c;\BL^2(\Omega)).$ It follows from weak--closedness of $\BU_0^\eta$ that $\Phi \in \BU_0^\eta$. Let $\tilde \Phi_n = \Phi_n + \Phi^\mathrm{e}$ and $\tilde \Phi = \Phi + \Phi^\mathrm{e}.$
	
	The corresponding subsequence $(\BS({\Phi_n}))$ is uniformly bounded in $\BBX_{T_c}(\BH^1(\Omega),\BL^2(\Omega))$ and is such that $\B{u}_{\tilde\Phi_n} \coloneqq \BBS(\tilde\Phi_n) - \BB_1(\tilde\Phi_n)$ solves \eqref{abs1} with control $\tilde \Phi_n$. Recall the embeddings 
	\begin{equation}\label{embrec}
		\BBX_{T_c}(\BH^1(\Omega),\BL^2(\Omega)) \hookrightarrow  \BBY_{T_c}^2(\BH^1(\Omega),\BL^2(\Omega))\hookrightarrow L^2(0,T_c;\BL^2(\Omega)).
	\end{equation}
	It follows from the first embedding in \eqref{embrec} that the subsequence $(\BS(\Phi_n))$ is also uniformly bounded in $\BBY_{T_c}^2(\BH^1(\Omega),\BL^2(\Omega))$ and since this is a Hilbert space, by passing (if necessary) on to a further subsequence, we can assume that there exists $\B{w}$ such that $\BS(\Phi_n) \rightharpoonup \B{w}$ in $\BBY_{T_c}^2(\BH^1(\Omega),\BL^2(\Omega))$. By the second embedding in \eqref{embrec} and by the uniqueness of the weak limit, we have $\B{w} = \B{v}.$
	
	We now show that $\B{v}$ is a weak solution of \eqref{abs1}. From the regularity of $\BBY_{T_c}^2(\BH^1(\Omega),\BL^2(\Omega))$, we only need to justify the passing to the limit in \eqref{weakform} for $\B{u}_n \coloneqq \BBS(\tilde\Phi_n) - \BB_1(\tilde\Phi_n)$, that is, 
	\begin{equation}\label{weakform3}
		\dfrac{d}{dt}(\B{u}_n(t,\cdot),\B{z})_{\BL^2(\Omega)} = (\B{u}_n(t,\cdot),(-\BA+\BP^\ast)\B{z})_{\BL^2(\Omega)} + (\BB\tilde\Phi_n(t),\B{z})_{\BL^2(\Omega)} + (\BF(\BBS(\tilde\Phi_n))(t,\cdot),\B{z})_{\BL^2(\Omega)}.
	\end{equation}
	It follows from the compactness of the embedding $H^2(0,T) \hookrightarrow H^1(0,T)$ that there is a (non--relabeled) subsequence $({\Phi_n})$ such that ${\Phi_n} \to \Phi$ in $[H^1(0,T_c)]^{2m}$ as $n \to \infty$. As a result $\BB_1\tilde\Phi_n \to \BB_1\tilde\Phi$ in $H^1(0,T;\BL^2(\Omega))$ and hence the linear terms above provide no problems. Hence it suffices to show that passing to the limit is justified for the nonlinear term. We are going to use the details given in proving \eqref{LipF0} in Lemma \ref{regF}. Again here we only address the case $m=1$. To that end we denote $\BBS(\tilde\Phi_n) = (u_1^n, v_1^n)$ and $\B{v} \in \BBY_{T_c}^2(\BH^1(\Omega),\BL^2(\Omega))$ as $\B{v} = (u,v).$ We then have (omitting the obvious details and sometimes the argument $(t,x)$)
	\begin{align*} 
		&
		\left|\left(\BF(\BBS(\tilde\Phi_n))(t,\cdot) - \BF(\B{v})(t,\cdot),\B{z}\right)_{\BL^2(\Omega)}\right|^2 
		\leqslant \left\|\BF(\BBS(\tilde\Phi_n))(t,\cdot) - \BF(\B{v})(t,\cdot)\right\|_{{\BL^2(\Omega)}}^2\|\B{z}\|_{\BL^2(\Omega)}^2 \\
		&=  D^2c^2\beta^2\left\|\frac{{u_1^n}|{u_1^n}|}{{v_1^n}}-\frac{{u}|{u}|}{{v}}\right\|_{L^2(\Omega)}^2 \|\B{z}\|_{\BL^2(\Omega)} \\
		&\lesssim \left[\int_0^{L }\big|{u_1^n} |{u_1^n} |({v} -{v_1^n} ) + ({u_1^n} (|{u_1^n} |-|{u} |)+({u_1^n} -{u} )|{u} |){v_1^n} \big|^2dx\right]\|\B{z}\|_{\BL^2(\Omega)}^2  \nonumber\\ 
		& \lesssim \left[\|{u_1^n} \|_{L^\infty(\Omega )}^4\|{v} -{v_1^n} \|_{L^2(\Omega )}^2 
		+ (\|{u_1^n} \|_{L^\infty(\Omega )}^2+\|{u} \|_{L^\infty(\Omega )}^2)\|{v_1^n} \|_{L^\infty(\Omega )}^2\|{u_1^n} -{u} \|_{L^2(\Omega )}^2\right]\|\B{z}\|_{\BL^2(\Omega)}^2  \nonumber\\ 
		& \lesssim \left[\|{v}(t,\cdot) -{v_1^n}(t,\cdot)\|_{L^2(\Omega )}^2 + \|{u_1^n}(t,\cdot) -{u}(t,\cdot)\|_{L^2(\Omega )}^2\right]\|\B{z}\|_{\BL^2(\Omega)}^2 \\
		&= \left\|\BBS(\tilde\Phi)(t,\cdot) - \B{v}(t,\cdot)\right\|_{\BL^2(\Omega)}^2\|\B{z}\|_{\BL^2(\Omega)}^2 \to 0
	\end{align*} 
	as $n \to \infty$ since the second embedding in \eqref{embrec} is compact.
	
	Now since the bounds of $\BBX_{T_c}^\BCK(\BH^1(\Omega),\BL^2(\Omega))$ on $\BS(\Phi_n)$ transfer to $\B{v}$ in $\BBY_{T_c}^2(\BH^1(\Omega),\BL^2(\Omega))$, we have that $\B{v}$ is a weak solution of \eqref{abs1} in $\mathfrak{X}_{T_c}^\BF$. But so is $\BS(\Phi)$. Therefore, $\BS(\Phi) = \B{v}$ by the uniqueness of weak solutions.
\end{proof}
In the next proposition we achieve our main goal, namely showing that $\BS: \BU_0^\eta \to L^2(0,T:\BL^2(\Omega))$ is weak--to--strong continuous. For this purpose, we need the following \emph{classical} lemma. 
\begin{lemma}[\bf Uryson's subsequence principle, \cite{T}]
	Let $(x_n)$ be a sequence in a topological space $X$, and let $x$ be another point in $X$. Then, the following are equivalent:
	\begin{itemize}
		\item[\bf (i)] $x_n \to x$ in $X.$ 
		\item[\bf (ii)] Every subsequence of $(x_n)$ has a further subsequence that converges to $x.$
	\end{itemize}
\end{lemma}
\begin{proposition}
	The map $\BS: \BU_0^\eta \to L^2(0,T; \BL^2(\Omega))$ is weak--to--strong continuous.
\end{proposition}
\begin{proof}
	Let $(\Phi_n)$ be a sequence in $\BU_0^\eta$ such that $\Phi_n \rightharpoonup \Phi$ in $\BU$. From weak closedness of $\BU_0^\eta$ it follows that $\Phi \in \BU_0^\eta$. We now claim that the sequence $(\BS(\Phi_n))$ converges to $\BS(\Phi)$ \emph{strongly} in $L^2(0,T;\BL^2(\Omega)).$
	
	Indeed, we know that $(\BS(\Phi_n))$ is uniformly bounded in $\BBY_{T_c}^2(\BH^1(\Omega),\BL^2(\Omega))$. Let $(\BS(\Phi_{n_k}))$ be a -- also uniformly bounded -- subsequence of $(\BS(\Phi_n))$. Then, there exists a further subsequence $(\BS(\Phi_{n_{k_j}}))$ which converges weakly to some $\B{v}$ in $\BBY_{T_c}^2(\BH^1(\Omega),\BL^2(\Omega))$. Hence, by compactness of the embedding $\BBY_{T_c}^2(\BH^1(\Omega),\BL^2(\Omega)) \hookrightarrow L^2(0,T;\BL^2(\Omega))$ there exists a further subsequence $(\BS(\Phi_{n_{k_{j_i}}}))$ that converges strongly to $\B{v}$ in $L^2(0,T;\BL^2(\Omega)).$ Now, weak closedness of $\BS$ guarantees that $\B{v} = \BS(\Phi)$. Therefore, Uryson's subsequence principle can be applied to yield $\BS(\Phi_n) \to \BS(\Phi)$ in $L^2(0,T;\BL^2(\Omega)).$
\end{proof}

\begin{remark}
	Notice that in the proof of the previous corollary the weak closedness of the map $\BS$ is used to tie the limit of any weakly convergent subsequence of $(\BS(\Phi_n))$ -- that in principle could converge to any function in $L^2(0,T;\BL^2(\Omega))$ -- to $\BS(\Phi).$
\end{remark}

\subsection{Adjoint states and differentiability of the control--to--state map}\label{adj_state_diff}

The following result is used to derive a formula for the gradient of $J.$ It follows by \cite[Theorem 1.2, p. 184]{P}.

\begin{proposition}\label{green22}
	Let $\B{f}, \B{g}: [0,T] \times \BL^2(\Omega) \to \BL^2(\Omega)$ be both continuous in $t$ and uniformly Lipschitsz continuous (w.r.t the second variable) in $\BL^2(\Omega)$. Then the abstract Cauchy problems 
	\begin{subnumcases}{\label{dder}}
		\B{w}_t = (\BA+\BP)\B{w}+\B{f}(t,\B{w}),\label{dder1} \\[2mm]
		\B{w}(0) = 0,\label{dder2} 
	\end{subnumcases} 
	and 
	\begin{subnumcases}{\label{dder3}}
		-\B{p}_t = (-\BA+\BP^\ast)\B{p}+\B{g}(t,\B{p}), \label{dder4} \\[2mm]
		\B{p}(T) = 0,\label{dder5} 
	\end{subnumcases} 
	have unique solutions $\B{w},\B{p} \in C([0,T],\BL^2(\Omega))$, respectively. Moreover, the following {Green}-type formula holds:
	\begin{equation}
		\label{green222}\int_0^T(\B{g}(t,\B{p}(t)),\B{w}(t))_{\BL^2(\Omega)}dt=  \int_0^T(\B{p}(t),\B{f}(t,\B{w}(t)))_{\BL^2(\Omega)}dt
	\end{equation}
\end{proposition}

For $\Phi \in \BU^{\mathrm{ad}}$ and $\BCK$ a suitable compact set, denote $\BBS(\Phi) = ((p^k(\Phi),q^k(\Phi))\kk{k} \in \BBX_{T_c}^\BCK(\BH^1(\Omega),\BL^2(\Omega))$ and define 
$$\calF^k(\Phi) \coloneqq 
\begin{bmatrix}
	0 & 0 \\ 
	\beta_k \dfrac{q^k(\Phi)|q^k(\Phi)|}{p^k(\Phi)^2} & -2\beta_k \dfrac{|q^k(\Phi)|}{p^k(\Phi)},
\end{bmatrix}$$ 
as well as $\BFF(\Phi) \coloneqq (\calF^k(\Phi))\kk{k}.$ A tedious but straighforward computation yields that $$\BFF(\Phi) \in \calL\left(\BV_T\right) \cap \calL\left(\BL^2(\Omega)\right)$$, and the inequalities 
\begin{equation}\label{Fineq} 
	\|\BFF(\Phi)\|_{\calL\left(\BV_T\right)} +\|\BFF(\Phi)\|_{\calL\left(\BL^2(\Omega)\right)} \leqslant C_\BCK\|\BS(\Phi)\|_{\BBX_{T_c}(\BH^1(\Omega),\BL^2(\Omega))},
\end{equation}
and 
\begin{equation}\label{Fineq2}
	\|\BFF(\Phi)-\BFF(\Psi)\|_{\calL\left(\BV_T\right)}+\|\BFF(\Phi)-\BFF(\Psi)\|_{\calL\left(\BL^2(\Omega)\right)} \leqslant C_\BCK\|\Phi-\Psi\|_{\BU} 
\end{equation}
hold for all $\Phi,\Psi \in \BU^{\mathrm{ad}}.$ 

Now, for a given $\Phi \in \BU_0^\eta$ and $\B{h} \in \BU_0$, we let $\tilde \Phi \coloneqq \Phi + \Phi^\mathrm{e}$ and define the affine map $\B{f}_\Phi(\B{h}): [0,T_c] \times \BL^2(\Omega) \to \BL^2(\Omega)$ by $\B{f}_\Phi(\B{h})(t,\B{w}) \coloneqq \BFF(\tilde \Phi)\B{w} + (\BB + \BFF(\tilde \Phi)\BB_1)\B{h}$. We notice that by the regularity of the operators involved in its definition we have that $\B{f}_\Phi(\B{h})$ is continuous in the first variable and (uniformly) Lipschitz continuous w.r.t the second variable. It then follows from Proposition \ref{green22} that there exists a unique solution $\B{w} \coloneqq \B{w}_\Phi(\B{h}) \in C([0,T_c];\BL^2(\Omega))$ to the problem  
\begin{subnumcases}{\label{dernn1}}
	\B{w}_t = (\BA+\BP)\B{w}+\B{f}_\Phi(\B{h})(t,\B{w}),\label{dernna} \\[2mm]
	\B{w}(0) = 0.\label{dernnb} 
\end{subnumcases} 
Now, let $\BE_\Phi: \BU_0 \to C([0,T_c];\BL^2(\Omega))$ be the operator that maps each $\B{h} \in \BU_0$ to $\BE_\Phi(\B{h}) = \B{w}_\Phi(\B{h}) + \BB_1\B{h}$. The main theorem of this section establishes that $\BE_\Phi$ characterizes the derivative of $\BS.$ Before proving this result, we establish properties of $\BE_\Phi.$
\begin{proposition}
	For each $\Phi \in \BU_0^\eta$, the operator $\BE_\Phi$ is linear and belongs to $\calL(\BU_0,L^2(0,T_c;\BL^2(\Omega))).$
\end{proposition}
\begin{proof}
	Linearity follows by linearity of $\BFF(\tilde\Phi)$ w.r.t $\B{h}$ along with the uniqueness of the solution to problem \eqref{dern1}. For continuity, first notice via semigroup theory that we have 
	\begin{align*}
		\|\B{w}_\Phi(\B{h})(t,\cdot)\|_{\BL^2(\Omega)} 
		&\leqslant \int_0^t \left\|\BT(t-\tau)\left[\BFF(\tilde\Phi)\B{w}_\Phi(\B{h})(\tau,\cdot) + (\BB + \BFF(\tilde\Phi)\BB_1)\B{h}(\tau,\cdot)\right]\right\|_{\BL^2(\Omega)}d\tau \\
		&\lesssim \|\BFF(\tilde\Phi)\|_{\calL(\BL^2(\Omega))}\int_0^{T_c}\left\|\B{w}_\Phi(\B{h})(\tau,\cdot)\right\|_{\BL^2(\Omega)}d\tau + (\|\BB\|+\|\BFF(\tilde\Phi)\|_{\calL(\BL^2(\Omega))}\|\BB_1\|)\|\B{h}\|_\BU,
	\end{align*} hence by Gr\"ownwall's inequality we have,
	\begin{align*}
		\|\BE_\Phi(\B{h})\|_{L^2(0,T_c;\BL^2(\Omega))}^2 
		= \int_0^{T_c} \|\B{w}_\Phi(\B{h})(t,\cdot) + \BB_1\B{h}(t,\cdot)\|_{\BL^2(\Omega)}^2dt 
		\lesssim \|\BB_1\|^2\|\B{h}\|_\BU^2 + C_{T_c} \|\B{h}\|_\BU^2 
		\lesssim \|\B{h}\|_{\BU}^2,
	\end{align*}
	which finishes the proof and establishes that $\|\BE_\Phi\|_{\calL(\BU_0,L^2(0,T_c;\BL^2(\Omega))} \lesssim C_{\BCK,T_c}$.
\end{proof}
\begin{proposition}\label{lipcontE}
	The map $\Phi \to \BE_\Phi$ is (Lipschitz) continuous from $\BU_0^\eta$ to $\calL(\BU_0,L^2(0,T_c;\BL^2(\Omega))).$
\end{proposition}
\begin{proof}
	Let $\Phi,\Psi \in \BU_0^\eta.$ By Gr\"ownwall's inequality applied at the difference $\B{w}_\Phi(\B{h}) - \B{w}_\Psi(\B{h})$ one obtains
	\begin{align*}
		&\|\BE_\Phi - \BE_\Psi\|_{\calL(\BU_0,L^2(0,T_c;\BL^2(\Omega)))} 
		= \sup_{\B{h} \in \BU_0, \|\B{h}\|_\BU = 1} \|\BE_\Phi(\B{h}) - \BE_\Psi(\B{h})\|_{L^2(0,T_c;\BL^2(\Omega)))}\\
		& = \sup_{\B{h} \in \BU_0, \|\B{h}\|_\BU = 1}\ \|\B{w}_\Phi(\B{h}) - \B{w}_\Psi(\B{h})\|_{L^2(0,T_c;\BL^2(\Omega)))} \\
		& = \sup_{\B{h} \in \BU_0, \|\B{h}\|_\BU = 1} \sqrt{\int_0^{T_c}\|\B{w}_\Phi(\B{h})(t,\cdot) - \B{w}_\Psi(\B{h})(t,\cdot)\|_{\BL^2(\Omega)}^2dt} \\
		& \lesssim \sup_{\B{h} \in \BU_0, \|\B{h}\|_\BU = 1} \sqrt{\int_0^{T_c}\|\B{h}\|_\BU^2\|\|(\BFF(\tilde\Phi)-\BFF(\tilde\Psi))\BE_\Phi(\B{h})(t,\cdot)\|_{\BL^2(\Omega)}^2dt} \\
		& \lesssim \|\Phi-\Psi\|_\BU\sup_{\B{h} \in \BU_0, \|\B{h}\|_\BU = 1} \sqrt{\int_0^{T_c}\|\BE_\Phi(\B{h})(t,\cdot)\|_{\BL^2(\Omega)}^2dt} \\
		& \lesssim \|\Phi-\Psi\|_\BU\sup_{\B{h} \in \BU_0, \|\B{h}\|_\BU = 1} \|\BE_\Phi(\B{h})\|_{L^2(0,T_c;\BL^2(\Omega))} 
		 = \|\Phi-\Psi\|_\BU\|\BE_\Phi\|_{\calL(\BU_0,L^2(0,T_c;\BL^2(\Omega))}  \lesssim \|\Phi-\Psi\|_\BU,
	\end{align*}
	which finishes the proof.
\end{proof}

In the next theorem, we assume without loss of generality that $\BU_0^\eta \subset \BU^{\mathrm{op}} \subset \BU_0$ where $\BU^{\mathrm{op}}$ is an open set on which $\BS$ is still well defined. This set can, for example, be constructed by defining a new $\BU_0^\eta$ replacing $\eta$ by $ \eta - \varepsilon$ and $\kappa_\BU$ by $\kappa_\BU - \varepsilon$ for a very small, but fixed $\varepsilon>0$. We then define $\BU^\mathrm{op} := \phi^{-1}(\mathrm{int}(\BU_0^\eta)$, with $\BU_0^\eta$ here being the original set.

\begin{theorem}\label{diffe}
	The maps $\BS: \BU^{\mathrm{op}} \to L^2(0,T_c;\BL^2(\Omega))$ and $\BS: \BU^{\mathrm{op}} \to \BV_{T_c}$ are continuously G\^{a}teaux--differentiable. Moreover, for each $\Phi \in \BU^\mathrm{op}$ we have
	\begin{equation}\label{derivative}
		\BBS'(\Phi,\cdot)  = \BE_{\Phi}(\cdot).
	\end{equation}
\end{theorem}
\begin{proof}
	This amounts to show that, given $\Phi \in \BU^\mathrm{op}$ and an arbitrary (but fixed) $\B{h} \in \BU_0$, the limit 
	\begin{equation}\label{derlim1}
		\lim\limits_{t \downarrow 0} \dfrac{\BS(\Phi+t\B{h})-\BS(\Phi)}{t}
	\end{equation}
	exists both in $L^2(0,T_c;\BL^2(\Omega)$ and $\BV_{T_c}$. Let $(t_n)$ be a sequence in $[0,1]$ such that $\Phi + t_n\B{h} \in \BU^{\mathrm{op}}$ and $t_n \downarrow 0.$ Consider a (non--relabelled) subsquence. Then it follows from Lipschitz continuity of $\BS$ (in $\BBX_{T_c}(\BH^1(\Omega),\BL^2(\Omega))$) that
	\begin{align*}
		\dfrac{1}{t_n} \|\BS(\Phi+t_n\B{h})-\BS(\Phi)\|_{\BBX_{T_c}(\BH^1(\Omega),\BL^2(\Omega))} 
		&\leqslant \dfrac{C_{T_c}}{t_n} \|t_n\B{h}\|_{\BU} = C_{T_c}\|\B{h}\|_\BU,
	\end{align*}
	rendering the subsequence $\left(\dfrac{\BS(\Phi+t_n\B{h})-\BS(\Phi)}{t_n}\right)_n$ uniformly (in $n$) bounded in $\BBX_{T_c}(\BH^1(\Omega),\BL^2(\Omega)).$ By compactness of the embedding $\BBY_{T_c}^\infty(\BH^1(\Omega),\BL^2(\Omega))\hookrightarrow \BV_{T_c}$ 
	it follows that, on a further (non--relabeled) subsequence, as $n \to \infty$ 
	\begin{equation}\label{strongS}
		\dfrac{\BS(\Phi+t_n\B{h})-\BS(\Phi)}{t_n} \to \overline{\B{v}}_\Phi(\B{h}) \in L^2(0,T_c;\BL^2(\Omega)) \cap \BV_{T_c}
	\end{equation} 
	in the strong sense, for some $\overline{\B{v}}_\Phi(\B{h}) \coloneqq ((\overline p^k_\Phi(\B{h}), \overline q^k_\Phi(\B{h})))\kk{k}$. We now show that $\overline{\B{v}}_\Phi(\B{h})$ is characterized uniquely by the solution of a PDE. 
	
	In fact, with reference to the map $\BF(\cdot)$ defined as $\BF(\cdot) = (\BBF^k(\cdot))\kk{k}$ (see \eqref{begF}) one can show by virtue of the strong convergence in \eqref{strongS} that, as $n \to \infty$, 
	\begin{equation}
		\dfrac{\BF(\BS(\Phi+t_n\B{h}))-\BF(\BS(\Phi))}{t_n} \to \BFF(\tilde\Phi)\overline{\B{v}}_\Phi(\B{h}), \qquad \tilde \Phi = \Phi + \Phi^\mathrm{e}
	\end{equation} 
	strongly in $L^2(0,T;\BL^2(\Omega))\cap \BV_{T_c}.$  Now, let $\B{u}_\Phi \coloneqq \BBS(\tilde\Phi) -\BB_1(\tilde\Phi)$ and notice that \eqref{abs1} implies that $$\B{w}_\Phi^n \coloneqq \dfrac{\B{u}_{\Phi+t_n\B{h}}-\B{u}_\Phi}{t_n}$$ solves 
	\begin{subnumcases}{\label{absder3}}
		{\B{w}_\Phi^n}_t = (\BA+\BP)\B{w}_\Phi^n + \dfrac{\BB(\tilde \Phi + t_n\B{h})-\BB(\tilde\Phi)}{t_n}+ \dfrac{\BF(\BBS(\tilde\Phi+t_n\B{h}))-\BF(\BBS(\tilde\Phi))}{t_n}, \label{lindera3} \\[2mm]
		{\B{w}_\Phi^n}(0) = -\dfrac{\BB_1(\tilde\Phi(0) + t_n\B{h}(0))-\BB_1(\tilde\Phi(0))}{t_n} = 0,\label{linderb3} 
	\end{subnumcases} 
	or equivalently, it satisfies the following implicit variation of parameters formula 
	\begin{align*}
		\B{w}_\Phi^n(t) & =  \int_0^t \BT(t-\tau)\left[ \dfrac{\BB(\tilde\Phi + t_n\B{h})-\BB(\tilde\Phi)}{t_n}+ \dfrac{\BF(\BBS(\tilde\Phi+t_n\B{h}))-\BF(\BBS(\tilde\Phi))}{t_n}\right](\tau)d\tau.
	\end{align*} 
	Then, as $n\to \infty$ above we have the following operator variation of parameter formula 
	\begin{align}\label{varder}
		\overline{\B{v}}_\Phi(\B{h}) - \BB_1\B{h} & =  \int_0^t\BT(t-\tau)\left[\BB\B{h} + \BFF(\tilde\Phi)\overline{\B{v}}_\Phi(\B{h})\right](\tau)d\tau.
	\end{align}
	But, again via semigroup theory, we see that \eqref{varder} implies that $\B{z}_\Phi(\B{h}) \coloneqq \overline{\B{v}}_\Phi(\B{h}) - \BB_1\B{h}$ solves 
	\begin{subnumcases}{\label{absder4}}
		{\B{z}_\Phi(\B{h})}_t = (\BA+\BP)\B{z}_\Phi(\B{h}) + \BB\B{h} + \BFF(\tilde\Phi)\overline{\B{v}}_\Phi({\B{h}}),\label{lindera4} \\[2mm]
		{\B{z}_\Phi(\B{h})}(0) = 0,\label{linderb4} 
	\end{subnumcases} 
	which can be rewritten equivalently as 
	\begin{subnumcases}{\label{absder5}}
		{\B{z}_\Phi(\B{h})}_t = (\BA+\BP)\B{z}_\Phi(\B{h}) + \B{f}_{\Phi}(\B{h}), \label{lindera5} \\[2mm]
		{\B{z}_\Phi(\B{h})}(0) = 0.\label{linderb5} 
	\end{subnumcases} Therefore, uniqueness of the solution to the problem \eqref{dern1} implies $\overline{\B{v}}_\Phi(\B{h}) - \BB_1\B{h} = \B{z}_\Phi \equiv \B{w}_{\Phi}(\B{h}),$ and hence $\overline{\B{v}}_\Phi(\B{h}) = \BE_{\Phi}(\B{h}).$
	
	We have then established the following:
	\begin{itemize}
			\item For any given subsequence of a given sequence $(t_n)$ in $[0,1]$ such that $\Phi + t_n\B{h} \in \BU^{\mathrm{op}}$ and $t_n \downarrow 0$ as $n \to \infty$ there exists a further subsequence such that $$\dfrac{\BS(\Phi+t_n\B{h})-\BS(\Phi)}{t_n} \to \overline{\B{v}}_\Phi({\B{h}}) \in L^2(0,T_c;\BL^2(\Omega)) \cap \BV_{T_c},$$ and this limit, in principle, depends on the chosen subsequence.  
			\item We shown, however, that $\overline{\B{v}}_\Phi(\B{h})$ is characterized uniquely by the \emph{unique} solution of a PDE, and therefore is also uniquely determined and does not depend of the sequence $(t_n).$ 
			\item Then it follows from the Uryson's subsequence principle that the limit 
			\begin{equation}\label{derlim}
				\lim\limits_{t_n \downarrow 0} \dfrac{\BS(\Phi+t_n \B{h})-\BS(\Phi)}{t_n}
			\end{equation}exists in $L^2(0,T_c;\BL^2(\Omega)) \cap \BV_{T_c}$ and coincide with $\BE_{\Phi}(\B{h}).$
			\item Putting everything together we conclude that $\BS$ is Gateaux differentiable on $\BU^\mathrm{op}$ and that for each $\Phi \in \BU^\mathrm{op}$, $\BS'(\Phi,\cdot) \in \calL(\BU_0,L^2(0,T_c;\BL^2(\Omega)) \cap \BV_{T_c})$ is given by $$\BS'(\Phi,\cdot) = \BE_{\Phi}(\cdot)$$ which we know is linear and bounded and, due to Lemma \ref{lipcontE}, also (Lipschitz) continuous. 
	\end{itemize}
	
	Therefore, $\BS$ is continuously Gateaux differentiable with the same derivative formula.
\end{proof}

Now, for $\Phi \in \BU^{\mathrm{ad}}$ and $\BCK$ a suitable compact set define $\BFF_\square(\Phi) \coloneqq \BFF(\Phi)^\ast$ when the latter is taken as a (bounded) operator from $\BL^2(\Omega)$ to itself. Similar to $\BFF(\Phi)$ in the previous section, a straightforward computation yields that $$\BFF_\square(\Phi) \in \calL\left(\BV_{T_c}\right) \cap \calL\left(\BL^2(\Omega)\right)$$, and the inequalities 
\begin{equation}\label{Fineqn} 
	\|\BFF_\square(\Phi)\|_{\calL\left(\BV_{T_c}\right)} +\|\BFF_\square(\Phi)\|_{\calL\left(\BL^2(\Omega)\right)} \leqslant C_\BCK\|\BS(\Phi)\|_{\BBX_{T_c}(\BH^1(\Omega),\BL^2(\Omega))}
\end{equation}
and 
\begin{equation}\label{Fineq2n}
	\|\BFF_\square(\Phi)-\BFF_\square(\Psi)\|_{\calL\left(\BV_{T_c}\right)}+\|\BFF_\square(\Phi)-\BFF_\square(\Psi)\|_{\calL\left(\BL^2(\Omega)\right)} \leqslant C_\BCK\|\Phi-\Psi\|_{\BU} 
\end{equation}
hold for all $\Phi,\Psi \in \BU^{\mathrm{ad}}.$ Moreover, for every $\B{u},\B{v} \in L^2(0,T_c;\BL^2(\Omega))$ we have 
\begin{equation}\label{adj11}
	(\BFF_\square(\Phi)\B{u},\B{v})_{L^2(0,T_c;\BL^2(\Omega))} = (\B{u},\BFF(\Phi)\B{v})_{L^2(0,T_c;\BL^2(\Omega))}.
\end{equation}
As a result, for a given $\Phi \in \BU_0^\eta$ and $\B{h} \in \BU_0$, we define the affine map $\B{g}_\Phi: [0,T_c] \times \BL^2(\Omega) \to \BL^2(\Omega)$ as $\B{g}_\Phi(t,\B{p}) \coloneqq \BFF_\square(\tilde\Phi)\B{p} + \BS(\Phi) - \B{v}_d$ and notice that, by the regularity of the operators involved in its definition, we have that $\B{g}_\Phi$ is continuous in the first variable and (uniformly) Lipschitz continuous w.r.t the second variable. It then follows from Proposition \ref{green22} that there exists a unique solution $\B{p} \coloneqq \B{p}_\Phi \in C([0,T_c];\BL^2(\Omega))$ to the problem
\begin{subnumcases}{\label{ddern6}}
	-\B{p}_t = (-\BA+\BP^\ast)\B{p}+\B{g}_\Phi(t,\B{p}), \label{ddern7} \\[2mm]
	\B{p}(T) = 0.\label{ddern8} 
\end{subnumcases} 
The next corollary characterizes the derivative $J'(\Phi)$ using the above function $\B{p}_\Phi$, hereafter called the \emph{adjoint} state.
\begin{corollary}\label{gradient}
	Let $\Phi \in \BU_0^\eta$ and $\BCK$ be a suitable compact set. With $\B{p}_\Phi \in C([0,T_c];\BL^2(\Omega))$ denoting the adjoint state, the following formula holds
	\begin{equation}\label{jgrad}
		J'(\Phi) = (\BB^\ast + \BB_1^\ast\BFF_\square(\tilde\Phi))\B{p}_\Phi - \BB_1^\ast(\BS(\Phi) - \B{v}_d) + \sigma \Phi,
	\end{equation}
	where the adjoints $\BB^\ast$ and $\BB_1$ are taken when $\BB, \BB_1$ are considered as bounded operators from $\BU^{\mathrm{ad}}$ to $L^2(0,T_c;\BL^2(\Omega)).$ 
\end{corollary}
\begin{proof}
	Recall that $\BS'(\Phi,\B{h}) = \BE_\Phi(\B{h}) = \B{w}_\Phi(\B{h}) + \BB_1\B{h}$. This along with the Green-type formula \eqref{green222} and \eqref{adj11} yields
	\begin{align*} 
		(J'(\Phi),\B{h})_\BU 
		& = (\BS(\Phi) - \B{v}_d,\BS'(\Phi)\B{h})_{L^2(0,T_c;\BL^2(\Omega))}+ \sigma(\Phi,\B{h})_{\BU} \\
		& = (\B{g}_\Phi(t,\B{p}_\Phi) - \BFF_\square(\tilde\Phi)\B{p}_\Phi, \B{w}_\Phi(\B{h}) )_{L^2(0,T_c;\BL^2(\Omega))} - (\BB_1^\ast(\BS(\Phi) - \B{v}_d) - \sigma \Phi,\B{h})_{\BU} \\
		& = (\B{p}_\Phi,\B{f}_\Phi(\B{h})(t,\B{w}_\Phi(\B{h}))-\BFF(\tilde\Phi)\B{w}_\Phi(\B{h}) )_{L^2(0,T_c;\BL^2(\Omega))}- (\BB_1^\ast(\BS(\Phi) - \B{v}_d) - \sigma \Phi,\B{h})_{\BU} \\
		& = (\B{p}_\Phi,(\BB + \BFF(\tilde\Phi)\BB_1)\B{h})_{L^2(0,T_c;\BL^2(\Omega))}- (\BB_1^\ast(\BS(\Phi) - \B{v}_d) - \sigma \Phi,\B{h})_{\BU} \\
		& = ((\BB^\ast + \BB_1^\ast\BFF_\square(\tilde\Phi))\B{p}_\Phi - \BB_1^\ast(\BS(\Phi) - \B{v}_d) + \sigma \Phi,\B{h})_{\BU},
	\end{align*}
	for all $\B{h} \in \BU_0.$	
\end{proof}

\bibliographystyle{abbrvurl} 
\bibliography{ref.bib}	
\end{document}